\documentclass[11pt]{amsart}
\usepackage[T1]{fontenc}
\usepackage[utf8]{inputenc}
\usepackage{amsmath,amssymb,amsthm,mathtools}
\usepackage{enumitem}
\usepackage[expansion=false]{microtype} 
\newtheorem{theorem}{Theorem}[section]
\newtheorem{lemma}[theorem]{Lemma}
\newtheorem{proposition}[theorem]{Proposition}
\newtheorem{corollary}[theorem]{Corollary}
\newtheorem{problem}[theorem]{Problem}
\theoremstyle{definition}
\newtheorem{definition}[theorem]{Definition}
\theoremstyle{remark}
\newtheorem{remark}[theorem]{Remark}

\newcommand{\R}{\mathbb{R}}
\newcommand{\Z}{\mathbb{Z}}
\newcommand{\Q}{\mathbb{Q}}
\newcommand{\N}{\mathbb{N}}
\newcommand{\T}{\mathbb{T}}
\newcommand{\Sph}{\mathbb{S}}
\newcommand{\dc}{d_{\mathrm c}}
\newcommand{\dist}{\operatorname{dist}}
\newcommand{\Leb}{\operatorname{Leb}}
\newcommand{\supp}{\operatorname{supp}}
\newcommand{\intr}{\operatorname{int}}

\begin{document}

\title[Quantitative spherical means along equidistributed spirals]{Quantitative time-averaged spherical means along equidistributed spirals: Diophantine rates and limits of uniformity}
\author{Claudio Fragomeli}
\date{}
\address{Independent researcher, Roccella Jonica (RC), Italy}
\email{cla.fragomeli@gmail.com}

\begin{abstract}
We establish quantitative equidistribution estimates for spherical observables sampled along Kronecker flows. For the monotone measure-preserving inverse-CDF parametrization $\Phi_d\colon\T^d\to\Sph^d$ and a Diophantine frequency of exponent $\tau$, every $\varphi\in C^s(\Sph^d)$, $0<s\le1$, has time averages converging to the spherical mean at rate $O(T^{-\vartheta})$, where
\[
\vartheta=\frac{s}{(1+s)(\tau+d)},
\]
uniformly in the initial phase and averaging window. Since $\Phi_d$ is discontinuous across a polar cut, the proof localizes the pullback away from the polar degeneracies and combines de la Vall\'ee Poussin approximation, Fourier decay, and the Diophantine lower bound. The same exponent holds for a continuous measure-preserving variant.

For comparison, mollification followed by Koksma--Hlawka gives, for every $\varepsilon>0$, the rate $O(T^{-s/[d(\tau+1)]+\varepsilon})$. Its power exponent is larger exactly when $\tau(d-1-s)<ds$; at equality the nominal exponents coincide, while the localization estimate has no $T^\varepsilon$ loss. On data vanishing near the polar degeneracies, the cutoff-free argument gives exponent $s/(\tau+d)$. We apply the estimates to shrinking spiral-type trajectories, integrable angular data, and homogeneous singularities, including a contrast between power-law and exponential radial contraction. Finally, on $\Sph^2$, for every prescribed $R(T)\to0$ we construct a rationally independent frequency and a smooth mean-zero observable with $\|\varphi\|_{C^1}\le1$ for which the error is not $O(R(T))$.
\end{abstract}

\maketitle

\section{Introduction}\label{sec:intro}

Let $\Sph^d\subset\R^{d+1}$, $d\ge1$, be the unit sphere equipped with normalized rotation-invariant measure $\sigma_d$, and write
\[
\langle\varphi\rangle=\int_{\Sph^d}\varphi\,d\sigma_d
\]
for the spherical mean. Consider the linear flow
\begin{equation}\label{eq:kron}
\theta(t)=\theta_0+t\omega \pmod{\Z^d},\qquad \theta_0\in\T^d,\quad \omega\in\R^d,
\end{equation}
on $\T^d=\R^d/\Z^d$. It is uniquely ergodic precisely when
\[
k\cdot\omega\neq0\qquad(k\in\Z^d\setminus\{0\}),
\]
i.e. when the coordinates of $\omega$ are linearly independent over $\Q$. This condition is weaker than the corresponding one for the discrete translation $n\mapsto\theta_0+n\omega$, where $1,\omega_1,\dots,\omega_d$ must be linearly independent over $\Q$.

The purpose of this paper is to obtain effective spherical equidistribution from the flow \eqref{eq:kron} and to use it for deterministic sampling along shrinking curves. We work with an explicit measure-preserving map
\[
\Phi_d\colon\T^d\longrightarrow\Sph^d,\qquad (\Phi_d)_\#\lambda_d=\sigma_d,
\]
constructed recursively from inverse cumulative-distribution latitudes in Section~\ref{sec:param}. Thus for a spherical observable $\varphi$ one expects the orbit $t\mapsto\Phi_d(\theta_0+t\omega)$ to sample $\Sph^d$ according to surface measure. Qualitatively, if $\omega$ is rationally independent, then for every continuous $\varphi$,
\begin{equation}\label{eq:qualitative}
\frac1T\int_0^T\varphi\bigl(\Phi_d(\theta_0+t\omega)\bigr)\,dt\longrightarrow\langle\varphi\rangle,
\end{equation}
uniformly in $\theta_0$; see Remark~\ref{rem:continuous}. The quantitative problem is subtler. The monotone parametrization $\Phi_d$ used here is discontinuous across a polar cut when regarded as a map on the torus, and its regularity degenerates near the poles. Consequently, even when $\varphi$ is H\"older on the sphere, the pullback $\varphi\circ\Phi_d$ is not in general a continuous torus observable to which a standard quantitative Fourier argument can be applied directly. Remark~\ref{rem:continuousparam} records that continuity is not intrinsically incompatible with measure preservation, and Remark~\ref{rem:jump} shows that for $d=2$ the jump of the pullback $\varphi\circ\Phi_2$ can be removed by subtracting an explicit sawtooth (the map $\Phi_2$ itself must of course remain sphere-valued, so it is the composition, not the parametrization, that is corrected). The point is instead to quantify the singular geometry of the particular monotone parametrization underlying the sampling trajectory, and to do so uniformly in the dimension.

\subsection*{Main results}
Assume that $\omega$ satisfies the Diophantine condition
\[
|k\cdot\omega|\ge\gamma |k|_\infty^{-\tau}\qquad(k\in\Z^d\setminus\{0\})
\]
with $\tau\ge d-1$. Our main localization result, Theorem~\ref{thm:main}, states that for $0<s\le1$ and $\varphi\in C^s(\Sph^d)$,
\begin{equation}\label{eq:intro-main-rate}
\left|\frac1T\int_a^{a+T}\varphi\bigl(\Phi_d(\theta_0+t\omega)\bigr)\,dt-\langle\varphi\rangle\right|
\le C\|\varphi\|_{C^s}\,T^{-\vartheta},
\qquad
\vartheta=\frac{s}{(1+s)(\tau+d)},
\end{equation}
uniformly in the initial phase $\theta_0$ and in the starting point $a$ of the time window. The proof localizes away from the polar sets, approximates the localized pullback by de la Vall\'ee Poussin polynomials, and combines Fourier decay with the Diophantine lower bound for $k\cdot\omega$. The cutoff error is controlled by one-dimensional occupation estimates for the coordinates of the flow. The uniformity in $a$ costs nothing --- for a translation flow the window $[a,a+T]$ started at $\theta_0$ is the window $[0,T]$ started at $\theta_0+a\omega$, so it is a formal consequence of the uniformity in $\theta_0$, which is what the proof actually delivers (Remark~\ref{rem:window}) --- but it is what allows the estimate to be applied on consecutive mesoscopic blocks with a frozen radius, which is what Theorem~\ref{thm:spiralblock} does, and we therefore state it explicitly. The word ``main'' refers here to the localization theorem because it addresses the structural transport of $C^s$ data through the singular monotone parametrization, not because its exponent is claimed to be pointwise optimal in every parameter regime: Remark~\ref{rem:HK} gives a complementary mollified Koksma--Hlawka estimate which is stronger for part of the range, for instance at $d=2$, $s=1$.

Indeed, let
\begin{equation}\label{eq:spiral}
\Gamma(t)=x_0+g(t)\Phi_d(\theta_0+t\omega),\qquad g(t)\searrow0.
\end{equation}
For the monotone map $\Phi_d$, the angular factor --- and hence in general $\Gamma$ --- has jump discontinuities when the torus orbit crosses the polar cut. We use ``spiral'' here in this dynamical sampling sense. Replacing $\Phi_d$ by the continuous measure-preserving variant $\widetilde\Phi_d$ of Remark~\ref{rem:continuousparam} produces a genuinely continuous shrinking curve, and Corollary~\ref{cor:tilde} gives the same angular exponent for that variant.
Section~\ref{sec:spirals} assumes that the angular profiles $f_r=f(x_0+r\,\cdot\,)$ are uniformly $C^s$ and that the radial modulus is H\"older of exponent $\kappa$ with a degeneration $\min(r,r')^{-\lambda}$ at the centre. The two ranges of $\lambda$ behave differently. If $\lambda<\kappa$, the radial hypothesis already forces the family $(f_r)$ to converge uniformly to a single limiting profile $f_0\in C^s(\Sph^d)$ at the rate $r^{\kappa-\lambda}$ (Lemma~\ref{lem:f0}), so the radial and angular errors decouple completely and Theorem~\ref{thm:spiral} gives
\[
O\!\left(T^{-\vartheta}\right)+O\Bigl(\frac1T\int_0^Tg(t)^{\kappa-\lambda}\,dt\Bigr),
\]
that is, for $g(t)=r_0(1+t)^{-b}$,
\[
O\!\left(T^{-\vartheta}+T^{-\min(b(\kappa-\lambda),1)}\log(2+T)\right)
\]
with no loss whatsoever in the angular exponent and no compatibility condition between $\kappa$, $\lambda$ and $b$ (Corollary~\ref{cor:power}). If instead $\lambda\ge\kappa$ no limiting profile need exist --- $f(x_0+r\xi)=a(\xi)\cos\log\frac1r$ with $\langle a\rangle=0$ is an example --- and the radius has to be tracked: Theorem~\ref{thm:spiralblock} freezes it on mesoscopic time blocks and applies \eqref{eq:intro-main-rate} on each translated block, which is where the window uniformity is used. For the power law $g(t)=r_0(1+t)^{-b}$ set $\beta=\kappa(b+1)-b\lambda$. In the degenerate range one always has $\beta\le\kappa\le1$; when $\beta>0$, Corollary~\ref{cor:powerdeg} gives, besides the radial mean term, the rate $O\bigl(T^{-\vartheta\beta/(\kappa+\vartheta)}\bigr)$ if $0<\beta<1$, and the endpoint rate $O\bigl((\log(2+T)/T)^{\vartheta/(1+\vartheta)}\bigr)$ if $\beta=1$ (which forces $\kappa=\lambda=1$). For $\beta\le0$ the general block estimate remains valid, but this power-law optimization alone does not yield a decaying radial-freezing error. Thus only in the degenerate range can the convergence rate reflect a genuine interaction between the arithmetic of the angular flow and the radial regularity of the shrinking trajectory; in the non-degenerate range the two errors are simply added.

We also consider two complementary regimes. First, Theorem~\ref{thm:L1} treats merely integrable angular data: Birkhoff's theorem gives convergence for almost every initial phase, and Remark~\ref{rem:L1sharp} shows that the exceptional null set cannot in general be removed. For homogeneous singularities at the center, an Abelian argument transfers ordinary ergodic averages to naturally normalized weighted averages. Power-law weights preserve convergence, and preserve the Diophantine rate for H\"older angular data, whereas Proposition~\ref{prop:exp} exhibits smooth mean-zero data for which exponentially weighted averages fail to converge for every phase. This establishes a concrete contrast between polynomial and exponential radial contraction without claiming an exhaustive borderline classification; the latter is left open in Section~\ref{sec:open}.

Second, the arithmetic hypothesis in \eqref{eq:intro-main-rate} cannot be replaced by rational independence if one seeks a rate uniform over the frequency. Theorem~\ref{thm:opt} proves that on $\Sph^2$, given any prescribed function $R(T)\to0$, there exist an irrational $\beta$, the frequency $\omega=(1,\beta)$, and a smooth mean-zero observable $\varphi$ such that
\[
\limsup_{T\to\infty}\frac{|T^{-1}\int_0^T\varphi(\Phi_2(t\omega))\,dt|}{R(T)}=\infty.
\]
The construction uses continued-fraction resonances, produces a Liouville frequency, and can be normalized so that $\|\varphi\|_{C^1}\le1$. Hence no single convergence modulus is valid over the full class of rationally independent frequencies, even on a bounded Lipschitz class. This statement is deliberately weaker than asserting that every non-Diophantine frequency has arbitrarily slow convergence; no such claim is made here.

\subsection*{Relation to previous work}
Quantitative uniform distribution for linear flows on tori is classical and is closely tied to Diophantine approximation and discrepancy; standard background is provided by \cite{KuipersNiederreiter,DrmotaTichy,Cassels,Schmidt}. More recent results address, among other questions, bounded-remainder phenomena for continuous linear flows \cite{Borda} and quantitative filling times under Diophantine or truncated Diophantine conditions \cite{DumasFischler}. These results concern the toral flow itself. The issue here is the transport of quantitative equidistribution through a spherical parametrization whose monotone inverse-CDF coordinates have polar degeneracies and a torus cut. The localization argument is designed precisely to retain an explicit rate after this transport.

A closely related and very active line of work quantifies the speed of convergence of Weyl sums over Kronecker sequences. Colzani \cite{Colzani} studies $N^{-1}\sum_{n\le N}f(x+n\alpha)-\int_{\T^d}f$ and shows how the logarithmic factors produced by the Koksma--Hlawka inequality together with the discrepancy of $\{n\alpha\}$ can be removed for suitable classes of $f$; his Theorem~4 also gives an abstract no-uniform-rate mechanism for Kronecker translations in translation-invariant Banach spaces, and Remark~17 records the continuous-time analogues obtained by replacing Weyl sums with integrals along $t\alpha$. Colzani, Gariboldi and Monguzzi \cite{ColzaniGariboldiMonguzzi} treat general summability methods, proving both a metric and a deterministic (Diophantine) estimate and showing that means other than the arithmetic ones can converge faster, and \cite{Chalmoukis} develops the corresponding theory for shift operators. Quadrature rules and the distribution of points on manifolds, including spheres, are studied in \cite{BrandoliniEtAl}, and accelerated convergence of ergodic averages along quasiperiodic orbits in \cite{DasYorke}.

These results are not directly applicable to Theorem~\ref{thm:main}, for a reason that should be stated precisely. Those among them that are comparable with the present problem require the observable to possess some form of global regularity: a modulus of continuity, membership in a Sobolev or Besov class, or bounded Hardy--Krause variation. The observable here is $F=\varphi\circ\Phi_d$, and by Remark~\ref{rem:cut}(c) the two one-sided limits of $\Phi_d$ across the cut $\{u_m=0\}$, $2\le m\le d$, differ by $2\bigl(\prod_{m<\ell\le d}\rho_\ell(u_\ell)\bigr)e_{m+1}$; freezing the remaining latitudes at the equator makes them antipodal. Hence, unless $\varphi$ takes the same value at that antipodal pair --- for a fixed pair, a closed hyperplane of $C^s(\Sph^d)$ --- arbitrarily small torus displacements across the cut change $F$ by an amount bounded below, so $F$ lies in no H\"older class on $\T^d$ and the approximation scheme of Lemma~\ref{lem:vdp} cannot be applied to it as it stands. This is a property of the composition, not of $\Phi_d$ alone: for $\varphi$ vanishing near the polar degeneracies, $F$ is H\"older on the torus (Remark~\ref{rem:localization}).

The bounded-variation hypothesis behaves differently, and we do not claim an obstruction there. Functions of bounded Hardy--Krause variation need not be continuous, and a jump across a coordinate hyperplane --- exactly the geometry of the cut $P_d$ --- is compatible with finite variation. What fails on the scale $C^s$, $0<s\le1$, is only the sufficient criterion at our disposal, namely pointwise existence and integrability of the mixed derivative $\partial_{u_1}\cdots\partial_{u_d}F$, which already for $d=2$ involves second derivatives of $\varphi$. But this is easily circumvented by mollifying $\varphi$ first, and it is worth doing so explicitly rather than leaving the comparison qualitative. Remark~\ref{rem:HK} carries this out: mollification at scale $\epsilon$ makes the pullback of bounded Hardy--Krause variation, Koksma--Hlawka applies, and optimizing in $\epsilon$ against the discrepancy of the linear flow \cite{Borda,DumasFischler} gives, for every $\varepsilon>0$,
\[
\bigl|A_{t_0,T}(\varphi)-\langle\varphi\rangle\bigr|\lesssim\|\varphi\|_{C^s}\,T^{-\vartheta_{\mathrm{HK}}+\varepsilon},
\qquad\vartheta_{\mathrm{HK}}=\frac{s}{d(\tau+1)} .
\]
Comparing with $\vartheta$, one has $\vartheta>\vartheta_{\mathrm{HK}}$ if and only if $\tau(d-1-s)>ds$. Thus neither route dominates: Theorem~\ref{thm:main} is the stronger statement on rough data and in high dimension --- for every $d$ and $\tau$ it wins for all small enough $s$, and for all $s\in(0,1]$ once $\tau(d-2)>d$ --- while the mollified Koksma--Hlawka bound is stronger at the Lipschitz endpoint in dimension two. The two are also of different kinds: the localization uses nothing beyond the $C^s$ norm appearing in its conclusion, its constant depends only on $(d,\gamma,\tau,s)$, and it loses no $T^{\varepsilon}$.

Two further comparisons are recorded below. Measure preservation does not force discontinuity: Remark~\ref{rem:continuousparam} exhibits a continuous, globally $\frac1d$-H\"older measure-preserving variant $\widetilde\Phi_d$. That map traverses each latitude twice and so parametrizes a different trajectory; Corollary~\ref{cor:tilde} shows nonetheless that Theorem~\ref{thm:main} covers it with the same exponent $\vartheta$, which is better than the exponent $\frac{s}{d(\tau+d)}$ that its global $\frac1d$-H\"older regularity would give by itself. Second, for $d=2$ the jump across the cut is constant, since the product $\prod_{2<\ell\le2}\rho_\ell$ is empty; subtracting an explicit sawtooth from $F$ then restores torus-H\"older regularity and makes a cutoff-free Fourier argument available. Remark~\ref{rem:jump} carries this out and returns $\frac{s}{2(\tau+2)}$, equal to $\vartheta$ at $s=1$ and strictly worse below. For $d\ge3$ the jump depends on the remaining latitudes and is itself discontinuous across the remaining cuts, so the correction would have to be built recursively; we have not carried this out and do not claim that it must degrade.

The cost of the localization step is quantified in Remark~\ref{rem:localization}: on the subclass of data vanishing near the polar degeneracies, where no cutoff is needed, the same lemmas return $\frac{s}{\tau+d}$, so that $\vartheta$ is smaller by exactly the factor $1+s$, in every dimension. That factor is the price of a cutoff whose scale must be chosen as a function of $T$, and it is paid only because no decay of $\varphi$ towards the polar set is assumed. The exponent $\vartheta$ itself of course depends on $d$, through $\tau+d$, exactly as $\frac{s}{\tau+d}$ does; what does not depend on $d$ is the ratio between the two.

One further difference should be recorded: the averages considered here are continuous-time averages along a flow, rather than Weyl sums over a discrete orbit. The estimate is also stated uniformly in the position of the averaging window; as noted above, that uniformity is a formal consequence of uniformity in the initial phase and is not an additional property to be proved, and it is used only in the degenerate range of Section~\ref{sec:spirals}, where the radius must be frozen block by block.

There is also a substantial literature on discrepancy, spherical designs, and quasi-Monte Carlo integration on spheres; see the survey \cite{BrauchartGrabner}. In particular, Brauchart and Dick \cite{BrauchartDick} lift digital nets to $\Sph^2$ through an area-preserving map, while QMC designs achieve optimal-order integration error in Sobolev spaces on $\Sph^d$ \cite{BrauchartSaffSloanWomersley}. Spiral constructions in particular have a long history in that literature: the generalized spiral points of Rakhmanov, Saff and Zhou \cite{RakhmanovSaffZhou}, popularized in \cite{SaffKuijlaars} and used for star catalogues in \cite{Bauer}, place $N$ nodes along a curve winding from pole to pole in a way that keeps neighbouring turns at comparable distance. The kinship with the trajectory studied here is real --- the latitude of \eqref{eq:spiral} is likewise swept monotonically --- but the design principle is opposite: there the number of turns is chosen as a function of $N$ so as to equidistribute the nodes, whereas here the winding is imposed by the frequency $\omega$ and cannot be tuned. Those constructions allow the nodes to be designed for numerical integration. In the present problem the sampling points are constrained to lie on the image of a single linear orbit, and the averaging interval may start at an arbitrary time. The arithmetic properties of $\omega$ therefore enter the error estimate directly. Our bound is not intended to compete with optimal QMC rates for freely chosen point sets; it addresses a different, dynamically constrained sampling problem.

The no-uniform-rate statement of Section~\ref{sec:opt} must likewise be placed against a classical background, since the absence of a universal rate in ergodic theorems is a known general phenomenon. Krengel \cite{Krengel} proved that for every ergodic invertible measure-preserving transformation of $[0,1]$ and every null sequence of positive reals there is a \emph{continuous} function whose Birkhoff averages converge, almost everywhere, more slowly than the prescribed sequence; see also Kakutani and Petersen \cite{KakutaniPetersen}. For circle rotations and Kronecker sequences specifically, the growth of Birkhoff sums of H\"older data is analysed by Kochergin \cite{Kochergin}; Moshchevitin \cite{Moshchevitin} revisits an example of Poincar\'e, gives multidimensional formulations in terms of Diophantine exponents, and proves the non-existence of a universal continuous function; and Chebotarenko \cite{Chebotarenko} studies the $\liminf$ behaviour of the sums $\sum_{k=0}^{Q-1}f(k\theta+x)$ for continuous $f$, obtaining results that depend both on the regularity class of $f$ and on the Diophantine properties of $\theta$. These last three works concern the discrete Kronecker orbit and describe how small, or how large, Birkhoff sums can become; Theorem~\ref{thm:opt} is instead a $\limsup$ statement for continuous-time averages of a spherical observable, and is complementary to them.

Accordingly, Theorem~\ref{thm:opt} is not presented as a new abstract slow-convergence principle: such mechanisms are already available for fixed ergodic transformations, and Colzani's framework includes Kronecker translations and a continuous-time analogue. The theorem is weaker in that it does not assert arbitrarily slow convergence for a \emph{given} frequency: the frequency $\beta$ is constructed jointly with the observable and with the target modulus. Its specific contribution here is the simultaneous realization of the obstruction within the spherical model treated in the paper, at the prescribed phase $\theta_0=0$, by an observable that is $C^\infty$, mean zero, and normalized in precisely the norm in which Theorem~\ref{thm:main} measures its data, namely $\|\varphi\|_{C^1(\Sph^2)}\le1$ (Remark~\ref{rem:normconv}). In dimension $d=2$, therefore, Theorems~\ref{thm:main} and~\ref{thm:opt} bracket one and the same bounded class of observables in the following limited sense: the Diophantine class admits the upper rate of Theorem~\ref{thm:main}, whereas over all rationally independent frequencies no prescribed modulus is uniform. This is the sense in which the Diophantine hypothesis cannot simply be replaced by rational independence in a frequency-uniform statement. The comparison does not claim sharpness for any fixed frequency. It is genuinely two-dimensional: Theorem~\ref{thm:opt} is proved for $d=2$ and for the initial phase $\theta_0=0$ only, the building blocks and the continued-fraction mechanism being planar and the argument using the phase to fix the sign of the resonant contribution. The extension to $d\ge3$, and to an arbitrary phase, is left open in Section~\ref{sec:open}.

Within this framework, the primary contribution is the phase-uniform localization estimate \eqref{eq:intro-main-rate} for the singular monotone pullback (and, by Corollary~\ref{cor:tilde}, for its continuous variant), together with the quantitative account in Remarks~\ref{rem:localization}--\ref{rem:HK} of the cost and range of that method. A second, model-specific contribution is the smooth, $C^1$-normalized realization in Section~\ref{sec:opt} of the no-uniform-modulus obstruction within this spherical Kronecker setting. The shrinking-trajectory and weighted-singular-data results are applications of the first contribution. We do not claim that the exponent $\vartheta$ is sharp within a fixed Diophantine class; determining the optimal exponent is one of the open problems in Section~\ref{sec:open}.

\subsection*{Organization}
Section~\ref{sec:param} constructs $\Phi_d$ and proves its push-forward and regularity properties. Section~\ref{sec:lemmas} collects the Fourier, approximation, oscillatory-integral, occupation-time, and recurrence estimates used later. Section~\ref{sec:main} proves the quantitative window-uniform equidistribution theorem. Section~\ref{sec:spirals} derives quantitative shrinking-spiral means. Section~\ref{sec:rough} treats integrable angular data and normalized weighted averages for singular data. Section~\ref{sec:opt} proves the absence of a rate uniform over all rationally independent frequencies. Section~\ref{sec:open} lists open problems.

\subsection*{Notation}
We write $e(x)=e^{2\pi i x}$ and $|k|_\infty=\max_m|k_m|$ for $k\in\Z^d$. On $\T^d$ we use $\dist_{\T^d}(u,u')=\max_m\dist_{\T}(u_m,u'_m)$; the cut metric $\dc$ is introduced in Section~\ref{sec:param}. For $0<s\le1$ we write $[\cdot]_s$ for the H\"older seminorm and $\|\cdot\|_{C^s}=\|\cdot\|_\infty+[\cdot]_s$, using the torus metric unless otherwise stated and the Euclidean chordal metric on $\Sph^d$. Haar probability measure on $\T^d$ is denoted by $\lambda_d$, and Fourier coefficients are $\widehat F(k)=\int_{\T^d}F(u)e(-k\cdot u)\,du$. Constants $C,c$ may change from line to line; subscripts indicate dependence.

\section{Diophantine classes and a measure-preserving parametrization of the sphere}\label{sec:param}

\begin{definition}[Diophantine classes]\label{def:dioph}
For $\gamma\in(0,1]$ and $\tau\ge d-1$ let
\[
\mathcal D(\gamma,\tau)=\bigl\{\omega\in\R^d :\ |k\cdot\omega|\ge\gamma\,|k|_\infty^{-\tau}\ \text{for all }k\in\Z^d\setminus\{0\}\bigr\}.
\]
\end{definition}

Taking $k=e_m$ shows $|\omega_m|\ge\gamma$ for all $m$ whenever $\omega\in\mathcal D(\gamma,\tau)$. By a standard Borel--Cantelli argument (see \cite[Ch.~VII]{Cassels} or \cite{Schmidt}), for every $\tau>d-1$ almost every $\omega\in\R^d$ belongs to $\mathcal D(\gamma,\tau)$ for some $\gamma>0$: the exceptional set for fixed $k$ has measure $O(\gamma|k|_\infty^{-\tau-1})$ locally, and $\sum_{k\neq0}|k|_\infty^{-\tau-1}<\infty$ precisely when $\tau>d-1$.

\begin{definition}[The map $\Phi_d$]\label{def:map}
Set $\Phi_1(u_1)=(\cos2\pi u_1,\sin2\pi u_1)$. For $m\ge2$ define
\[
G_m\colon[-1,1]\longrightarrow[0,1],\qquad
G_m(z)=c_m\int_{-1}^{z}(1-v^2)^{\frac{m-2}2}\,dv,
\]
\[
c_m^{-1}=\int_{-1}^{1}(1-v^2)^{\frac{m-2}2}\,dv,
\]
which is a continuous increasing bijection; let $z_m=G_m^{-1}\colon[0,1]\to[-1,1]$ and $\rho_m(u)=\sqrt{1-z_m(u)^2}$. Recursively, for $u=(u_1,\dots,u_m)\in\T^m$,
\[
\Phi_m(u)=\bigl(\rho_m(u_m)\,\Phi_{m-1}(u_1,\dots,u_{m-1}),\ z_m(u_m)\bigr)\in\Sph^m\subset\R^{m+1}.
\]
(For $u_m\in\T$ we always use its unique representative in $[0,1)$, so the value on the cut is the one corresponding to $z_m(0)=-1$. When a closed fundamental cube $[0,1]^m$ is used for interval or variation arguments, we additionally use the natural one-sided interval extension $z_m(1)=+1$ and $\rho_m(1)=0$; this auxiliary boundary value represents the opposite side of the cut and is not part of the torus definition. Thus $\Phi_m$ is well defined as a Borel map on $\T^m$ and continuous off the polar set $P_m=\bigcup_{2\le\ell\le m}\{u_\ell=0\}$.)
\end{definition}

\begin{lemma}[Push-forward]\label{lem:push}
For every $d\ge1$, $(\Phi_d)_\#\lambda_d=\sigma_d$; equivalently, for every $\varphi\in C(\Sph^d)$,
\[
\int_{\T^d}\varphi(\Phi_d(u))\,du=\int_{\Sph^d}\varphi\,d\sigma_d.
\]
\end{lemma}

\begin{proof}
Induction on $d$. The case $d=1$ is the arc-length parametrization of the circle. For $d=m\ge2$ we use the slice (Fubini) formula for the normalized measure on $\Sph^m$ (see, e.g., \cite[Ch.~1]{AtkinsonHan}): for $\varphi\in C(\Sph^m)$,
\begin{equation}\label{eq:slice}
\int_{\Sph^m}\varphi\,d\sigma_m=c_m\int_{-1}^{1}(1-z^2)^{\frac{m-2}2}\int_{\Sph^{m-1}}\varphi\bigl(\sqrt{1-z^2}\,y,\,z\bigr)\,d\sigma_{m-1}(y)\,dz,
\end{equation}
with $c_m$ as above precisely because $\sigma_m$ is a probability measure. In \eqref{eq:slice} substitute $z=z_m(u_m)$, so that $du_m=G_m'(z)\,dz=c_m(1-z^2)^{\frac{m-2}2}\,dz$; the right-hand side becomes
\begin{align*}
&\int_0^1\!\int_{\Sph^{m-1}}
\varphi\bigl(\rho_m(u_m)y,\,z_m(u_m)\bigr)\,d\sigma_{m-1}(y)\,du_m\\
&\qquad=\int_0^1\!\int_{\T^{m-1}}
\varphi\bigl(\Phi_m(u)\bigr)\,du_1\cdots du_{m-1}\,du_m,
\end{align*}
where the inner identity is the inductive hypothesis applied, for each fixed $u_m$, to the continuous function $y\mapsto\varphi(\rho_m(u_m)y,z_m(u_m))$ on $\Sph^{m-1}$.
\end{proof}

\begin{remark}[Uniform latitude fails for $d\ge3$]\label{rem:latitude}
For $m=2$ one has $G_2(z)=\frac{1+z}2$, i.e.\ $z_2(u)=2u-1$: the latitude may be taken uniform in $u$, which is Archimedes' theorem and explains why the elementary uniform-latitude parametrization works on $\Sph^2$ (and trivially on $\Sph^1$, which has no latitude). For $m\ge3$, however, the latitudinal density $(1-z^2)^{(m-2)/2}$ is not constant. The uniform-latitude map obtained by taking $z=2u-1$ therefore pushes Lebesgue measure to a measure whose density with respect to $\sigma_m$ is proportional to $(1-z^2)^{-(m-2)/2}$; in particular it is not measure preserving. The inverse-CDF choice in Definition~\ref{def:map} is what produces the correct spherical measure in every dimension.
\end{remark}

\begin{remark}[A continuous measure-preserving variant]\label{rem:continuousparam}
The discontinuity of $\Phi_d$ on the torus is a feature of the monotone one-pass latitude coordinate, not an intrinsic consequence of measure preservation. Let $q\colon\T\to[0,1]$ be the tent map
\[
q(u)=\begin{cases}2u,&0\le u\le\frac12,\\ 2(1-u),&\frac12\le u<1.\end{cases}
\]
Then $q$ is continuous on $\T$ and $q_\#\lambda_1$ is Lebesgue measure on $[0,1]$. Define recursively $\widetilde\Phi_1=\Phi_1$ and
\[
\widetilde\Phi_m(u)=\bigl(\rho_m(q(u_m))\,\widetilde\Phi_{m-1}(u_1,\dots,u_{m-1}),\ z_m(q(u_m))\bigr).
\]
Because $q$ is $2$-Lipschitz and $q_\#\lambda_1$ is Lebesgue measure on $[0,1]$, induction together with the slice formula in Lemma~\ref{lem:push} shows that $\widetilde\Phi_d$ is continuous and $(\widetilde\Phi_d)_\#\lambda_d=\sigma_d$. The interval estimates \eqref{eq:interval} below, composed with $q$, give by the same recursive argument a global $1/d$-H\"older bound with respect to the torus metric; taking $u_d\to0$ shows that this exponent is sharp for this construction. We retain the monotone map $\Phi_d$ because it is the map for which the sampling problem is posed: its latitude coordinate is one-to-one on the cut fundamental domain, so that in each latitude a single pass of the flow through $[0,1)$ sweeps the sphere monotonically from pole to pole. The map $\widetilde\Phi_d$ traverses every latitude twice, once in each direction, and therefore describes a genuinely different trajectory; results about $\widetilde\Phi_d$ are results about a different sampling problem rather than weaker versions of those proved here. That said, the two problems are not in competition: Corollary~\ref{cor:tilde} shows that the proof of Theorem~\ref{thm:main} covers $\widetilde\Phi_d$ verbatim and with the same exponent $\vartheta$, which is better than what its global $\frac1d$-H\"older regularity alone would give. No argument below relies on discontinuity being unavoidable for all measure-preserving parametrizations.
\end{remark}

We now quantify the regularity of $\Phi_d$. A point deserves emphasis at the outset: since $z_m(0^+)=-1$ while $z_m(1^-)=+1$, the map $\Phi_d$ is genuinely discontinuous across the polar cut $P_d$ for $d\ge2$, and consequently no H\"older (indeed, no uniform-continuity) estimate can hold with respect to the torus metric $\dist_{\T^d}$; see Remark~\ref{rem:cut}. The correct quantitative statement is expressed in the cut metric: identifying $\T^d$ with the fundamental domain $[0,1)^d$, set
\begin{equation}\label{eq:cutmetric}
\dc(u,u'):=\max\Bigl(\dist_{\T}(u_1,u'_1),\ \max_{2\le m\le d}|u_m-u'_m|\Bigr),
\end{equation}
where the coordinates $u_m,u'_m$, $m\ge2$, are read through their representatives in $[0,1)$. Then $\dist_{\T^d}\le\dc\le1$, with $\dist_{\T^d}(u,u')=\dc(u,u')$ precisely when, for every $m\ge2$, the geodesic realizing $\dist_{\T}(u_m,u'_m)$ does not cross the cut $\{u_m=0\}$. Set, for $0<\eta<\frac14$,
\[
\T^d_\eta=\bigl\{u\in\T^d :\ \min_{2\le m\le d}\min(u_m,1-u_m)\ge\eta\bigr\}\qquad(\text{with }\T^1_\eta=\T^1).
\]

\begin{lemma}[H\"older and localized Lipschitz bounds]\label{lem:reg}
Let $d\ge1$. There is $C_d<\infty$ such that:
\begin{enumerate}[label=\textup{(\alph*)}, leftmargin=2.2em]
\item $|\Phi_d(u)-\Phi_d(u')|\le C_d\,\dc(u,u')^{1/d}$ for all $u,u'\in\T^d$;
\item for every $\eta\in(0,\frac14)$ and all $u,u'\in\T^d_\eta$,
\[
|\Phi_d(u)-\Phi_d(u')|\le C_d\,\eta^{\frac1d-1}\,\dc(u,u').
\]
\end{enumerate}
Moreover the exponent $1/d$ in \textup{(a)} cannot be improved for $d\ge2$. In particular, both bounds hold with $\dist_{\T^d}(u,u')$ on the right-hand side whenever the torus distance is realized without crossing the cut, i.e.\ whenever $\dist_{\T^d}(u,u')=\dc(u,u')$.
\end{lemma}

\begin{proof}
\emph{Step 1: one-dimensional bounds.} Fix $2\le m\le d$. For $z\in[-1,0]$ we have $1-v\in[1,2]$ on $[-1,z]$, hence $G_m(z)\asymp_m\int_{-1}^{z}(1+v)^{\frac{m-2}2}\,dv=\frac2m(1+z)^{m/2}$, so that
\begin{equation}\label{eq:onedim}
1+z_m(u)\asymp_m u^{2/m}\quad(0\le u\le\tfrac12),\qquad 1-z_m(u)\asymp_m (1-u)^{2/m}\quad(\tfrac12\le u\le1)
\end{equation}
by symmetry ($G_m(-z)=1-G_m(z)$). Writing $w(u)=\min(u,1-u)$, \eqref{eq:onedim} gives $1-z_m(u)^2=(1-z_m)(1+z_m)\asymp_m w(u)^{2/m}$ and $\rho_m(u)\asymp_m w(u)^{1/m}$ on $(0,1)$. Differentiating $G_m(z_m(u))=u$,
\begin{align*}
z_m'(u)
&=\frac1{G_m'(z_m(u))}
 =\frac1{c_m\,(1-z_m(u)^2)^{\frac{m-2}2}}
 \le C_m\,w(u)^{\frac2m-1},\\
|\rho_m'(u)|
&=\frac{|z_m(u)|\,z_m'(u)}{\rho_m(u)}
 \le C_m\,w(u)^{\frac1m-1}.
\end{align*}
We use the elementary fact that if $h\in C^1((0,1))$ satisfies $|h'(u)|\le M\,w(u)^{\alpha-1}$ with $\alpha\in(0,1]$, then $[h]_{C^\alpha([0,1])}\le2M/\alpha$. Indeed, for $0\le u<u'\le\frac12$,
\[
|h(u')-h(u)|\le M\int_u^{u'}v^{\alpha-1}\,dv
=\frac M\alpha(u'^\alpha-u^\alpha)
\le\frac M\alpha(u'-u)^\alpha,
\]
by subadditivity of $t\mapsto t^\alpha$. The case $u,u'\ge\frac12$ is symmetric, and a general pair $u<\frac12<u'$ is split at $\frac12$, which costs the factor $2$ since $(u'-\frac12)^\alpha+(\frac12-u)^\alpha\le2(u'-u)^\alpha$. Hence, with respect to the standard metric of the interval $[0,1]$,
\begin{equation}\label{eq:interval}
[z_m]_{C^{2/m}([0,1])}\le C_m,\qquad [\rho_m]_{C^{1/m}([0,1])}\le C_m,
\end{equation}
and on $[\eta,1-\eta]$ we have $w(u)\ge\eta$, so that both exponents can be compared with $\frac1d-1$: since $\frac1m-1\le0$ and $\frac1m\ge\frac1d$, one has $w(u)^{\frac1m-1}\le\eta^{\frac1m-1}\le\eta^{\frac1d-1}$, and since $\frac2m-1\le0$ for $m\ge2$ while $\frac2m\ge\frac2d\ge\frac1d$, also $w(u)^{\frac2m-1}\le\eta^{\frac2m-1}\le\eta^{\frac1d-1}$. Hence
\begin{equation}\label{eq:localder}
\sup_{[\eta,1-\eta]}\bigl(|z_m'|+|\rho_m'|\bigr)\le C_d\,\eta^{\frac1d-1}\qquad(2\le m\le d).
\end{equation}
We stress that \eqref{eq:interval} is an interval bound and does not extend to the torus metric on the coordinate $u_m$: the function $\rho_m$ takes the common value $0$ at both endpoints and does extend, but $z_m(0)=-1\neq1=z_m(1)$, so $z_m$ has no continuous $1$-periodic extension. This is the source of the discontinuity of $\Phi_m$ across $\{u_m=0\}$ and the reason why the cut metric \eqref{eq:cutmetric} is used throughout.

\emph{Step 2: induction for \textup{(a)}.} The case $d=1$ is clear ($\Phi_1$ is smooth and $1$-periodic, and $\dc=\dist_{\T}$ there). Let $m\ge2$ and assume (a) for $m-1$. With $\bar u=(u_1,\dots,u_{m-1})$ and $\delta=\dc(u,u')\le1$,
\begin{align*}
|\Phi_m(u)-\Phi_m(u')|
&\le \rho_m(u_m)\bigl|\Phi_{m-1}(\bar u)-\Phi_{m-1}(\bar u')\bigr|\\
&\quad+|\rho_m(u_m)-\rho_m(u'_m)|+|z_m(u_m)-z_m(u'_m)|\\
&\le C_{m-1}\delta^{\frac1{m-1}}+C_m\delta^{\frac1m}+C_m\delta^{\frac2m}\\
&\le C'_m\,\delta^{\frac1m},
\end{align*}
because $\rho_m\le1$, $|\Phi_{m-1}|=1$, and $\delta\le1$ makes the smallest exponent $\frac1m$ dominant. Here the coordinate $u_m$ is measured, in accordance with \eqref{eq:cutmetric}, through its representative in $[0,1)$, so that the interval bounds \eqref{eq:interval} apply directly; no periodic extension of $z_m$ or $\rho_m$ is invoked.

\emph{Step 3: \textup{(b)}.} On the product region $\T^m_\eta$ each coordinate segment joining $u$ to $u'$ coordinatewise (the segment of $[\eta,1-\eta]$ for the coordinates $2\le\ell\le m$, the torus geodesic for $u_1$) stays in $\T^m_\eta$; telescoping in the coordinates and using \eqref{eq:localder} for $\ell=2,\dots,m$, the smooth dependence on $u_1$ ($|\partial_{u_1}\Phi_m|\le2\pi$), and $\rho_\ell\le1$, gives $|\Phi_m(u)-\Phi_m(u')|\le C_m\eta^{\frac1d-1}\dc(u,u')$.

\emph{Sharpness.} Taking $u'=(u_1,\dots,u_{d-1},0)$ and $u_d=\varepsilon$, the last coordinate moves by $|z_d(\varepsilon)-z_d(0)|\asymp\varepsilon^{2/d}$ while $\rho_d(\varepsilon)\asymp\varepsilon^{1/d}$ moves the first block by $\asymp\varepsilon^{1/d}$; hence exponent $1/d$ is attained.
\end{proof}

\begin{remark}[No torus-metric regularity; geometry of the cut]\label{rem:cut}
(a) For $d\ge2$ no estimate of the form $|\Phi_d(u)-\Phi_d(u')|\le\Xi(\dist_{\T^d}(u,u'))$ with $\Xi(0^+)=0$ can hold: with $u=(u_1,\dots,u_{d-1},\varepsilon)$ and $u'=(u_1,\dots,u_{d-1},1-\varepsilon)$ one has $\dist_{\T^d}(u,u')=2\varepsilon\to0$ while $\Phi_d(u)$ and $\Phi_d(u')$ converge to the south and north pole respectively, so $|\Phi_d(u)-\Phi_d(u')|\to2$. The two metrics coincide exactly on pairs whose coordinate geodesics avoid the cut, which is the regime in which Lemma~\ref{lem:reg} is used in Section~\ref{sec:main}.

(b) Consequently $F=\varphi\circ\Phi_d$ is in general discontinuous on $\T^d$: across $\{u_m=0\}$, $m\ge2$, it jumps between the values of $\varphi$ at the two one-sided limits of $\Phi_d$, computed in \textup{(c)} below. On $\{u_m=0\}$ the first $m$ Euclidean coordinates of $\Phi_d$ vanish, and the same holds for all one-sided limits along the cut; hence the closure of $\Phi_d(P_d)$, together with all such limits, is contained in the great subsphere
\[
\Sigma:=\{x\in\Sph^d : x_1=x_2=0\}
\]
(for $d=2$, $\Sigma$ is the pair of poles). Thus $F$ is continuous on $\T^d$ whenever $\varphi$ vanishes on a neighbourhood of $\Sigma$ --- and then, by Remark~\ref{rem:localization}, even H\"older on $\T^d$, so that the failure of a modulus of continuity in \textup{(a)} is a property of $\Phi_d$ that transfers to $F$ only for suitable $\varphi$. For a general $\varphi$, the localization device of Section~\ref{sec:main} makes Fourier analysis available for this monotone pullback; see Remark~\ref{rem:localization}.

(c) \emph{The jump, and when it is antipodal.} Unwinding Definition~\ref{def:map}, for $2\le m\le d$ the first $m+1$ Euclidean coordinates of $\Phi_d(u)$ are
\[
\Bigl(\prod_{m<\ell\le d}\rho_\ell(u_\ell)\Bigr)\,\Phi_m(u_1,\dots,u_m),
\]
while for $m<j\le d$ the $(j+1)$-st coordinate is
\[
\Bigl(\prod_{j<\ell\le d}\rho_\ell(u_\ell)\Bigr)\,z_j(u_j)
\]
(empty products being equal to $1$, so that the last coordinate is $z_d(u_d)$). Moreover $\Phi_m(u_1,\dots,u_m)\to(0,\dots,0,-1)$ as $u_m\to0^+$, resp.\ $\to(0,\dots,0,+1)$ as $u_m\to1^-$, uniformly in $u_1,\dots,u_{m-1}$. Hence the two one-sided limits of $\Phi_d$ along $\{u_m=0\}$ differ by
\begin{equation}\label{eq:jumpvector}
2\Bigl(\prod_{m<\ell\le d}\rho_\ell(u_\ell)\Bigr)e_{m+1},
\end{equation}
so they are reflections of one another in the hyperplane $\{x_{m+1}=0\}$. They are \emph{antipodal} precisely when $\prod_{m<\ell\le d}\rho_\ell(u_\ell)=1$, i.e.\ when every remaining latitude sits at the equator, $u_\ell=\frac12$ (equivalently $z_\ell(u_\ell)=0$); for $d=2$ the product is empty and the two limits are always the poles $S$ and $N$. Conversely the jump \eqref{eq:jumpvector} degenerates as the remaining latitudes approach their own poles. Thus the oscillation of $F$ across the cut is bounded below only away from that degeneracy, and it is there --- e.g.\ at $u_\ell=\frac12$ for $\ell>m$ --- that one tests the absence of a torus modulus of continuity for $F$, as in Section~\ref{sec:intro}.
\end{remark}

\section{Preliminary estimates}\label{sec:lemmas}

\begin{lemma}[Fourier decay of H\"older functions]\label{lem:fourier}
Let $F\in C^s(\T^d)$, $0<s\le1$. Then for every $k\in\Z^d\setminus\{0\}$,
\[
|\widehat F(k)|\le\tfrac12\,[F]_s\,\bigl(2|k|_\infty\bigr)^{-s}.
\]
\end{lemma}

\begin{proof}
Pick $j$ with $|k_j|=|k|_\infty$ and let $h=\frac1{2k_j}e_j$. Translation invariance gives $\widehat F(k)=\int F(u+h)e(-k\cdot(u+h))\,du=-\int F(u+h)e(-k\cdot u)\,du$, whence $\widehat F(k)=\frac12\int\bigl(F(u)-F(u+h)\bigr)e(-k\cdot u)\,du$ and $|\widehat F(k)|\le\frac12[F]_s\dist_{\T^d}(0,h)^s=\frac12[F]_s(2|k|_\infty)^{-s}$.
\end{proof}

\begin{lemma}[de la Vall\'ee Poussin approximation]\label{lem:vdp}
For $K\in\N$ let $V_K=2\mathcal F_{2K+1}-\mathcal F_K$, where $\mathcal F_N$ denotes the Fej\'er mean of order $N$ on $\T^1$ (the letter $\sigma$ being reserved for the measures $\sigma_d$), and, for $1\le m\le d$, let $V_K^{[m]}$ denote the operator acting by $V_K$ in the $m$-th variable of a function on $\T^d$ and by the identity in the others, and $V_K^{\otimes d}=V_K^{[1]}\cdots V_K^{[d]}$ the product operator acting by $V_K$ in each of the $d$ variables. Then $P=V_K^{\otimes d}F$ is a trigonometric polynomial with spectrum in $\{|k|_\infty\le2K+1\}$ satisfying
\begin{align*}
\widehat P(k)&=v_K(k_1)\cdots v_K(k_d)\,\widehat F(k),\\
v_K(j)&\in[0,1],\qquad v_K(j)=1\ \text{for }|j|\le K+1,
\end{align*}
and there is $C_d<\infty$ such that for all $F\in C^s(\T^d)$, $0<s\le1$,
\[
\|F-V_K^{\otimes d}F\|_\infty\le C_d\,[F]_s\,K^{-s}.
\]
In particular $\widehat P(0)=\widehat F(0)$ and $|\widehat P(k)|\le|\widehat F(k)|$ for all $k$.
\end{lemma}

\begin{proof}
The Fej\'er multiplier is $(1-\frac{|j|}{N+1})_+$, so $v_K(j)=2(1-\frac{|j|}{2K+2})_+-(1-\frac{|j|}{K+1})_+$, which equals $1$ for $|j|\le K+1$, decreases linearly to $0$ on $K+1\le|j|\le2K+2$, and vanishes beyond; in particular $v_K\in[0,1]$. Since Fej\'er means are positive contractions on $C(\T^1)$, $\|V_K\|_{C\to C}\le3$; since $v_K(j)=1$ for $|j|\le K+1$, $V_K$ reproduces trigonometric polynomials of degree $\le K+1$, hence for $f\in C^s(\T^1)$
\[
\|f-V_Kf\|_\infty\le(1+\|V_K\|)\,E_{K+1}(f)\le4\,E_{K+1}(f)\le C\,[f]_s\,K^{-s}
\]
by Jackson's theorem \cite[Ch.~7]{DeVoreLorentz} (see also \cite[Ch.~III]{Zygmund}). For $d$ variables use the telescoping identity
\[
F-V_K^{\otimes d}F
=\sum_{m=1}^{d}V_K^{[1]}\cdots V_K^{[m-1]}\bigl(I-V_K^{[m]}\bigr)F,
\]
in the notation of the statement, the bracketed superscript indicating the variable acted upon. Each one-di\-men\-sion\-al section of $F$ in the $m$-th variable has $C^s$ seminorm at most $[F]_s$ (the metric on $\T^d$ dominates each coordinate metric), so $\|(I-V_K^{[m]})F\|_\infty\le C[F]_sK^{-s}$, and the outer operators contribute a factor $3^{m-1}$. Summing proves the claim with $C_d=C\,d\,3^{d-1}$.
\end{proof}

\begin{lemma}[Oscillatory integral]\label{lem:osc}
For every $a\neq0$, $t_0\in\R$ and $T>0$,
\[
\Bigl|\frac1T\int_{t_0}^{t_0+T}e(at)\,dt\Bigr|=\frac{|e(aT)-1|}{2\pi|a|T}\le\min\Bigl(1,\frac1{\pi|a|T}\Bigr).
\]
\end{lemma}

\begin{proof}
Direct computation; the modulus is independent of $t_0$.
\end{proof}

\begin{lemma}[Occupation times of circle rotations]\label{lem:occupation}
Let $\omega_*\neq0$, $\theta(t)=\theta_*+\omega_*t$ on $\T^1$, and let $I\subset\T^1$ be an arc. Then for all $t_0\in\R$, $T>0$,
\[
\Bigl|\frac1T\,\Leb\{t\in[t_0,t_0+T] : \theta(t)\in I\}-|I|\Bigr|\le\frac1{|\omega_*|T}.
\]
\end{lemma}

\begin{proof}
The motion is periodic with period $p=1/|\omega_*|$ and spends exactly $|I|p$ in $I$ during each full period. Write $T=Np+r$, $0\le r<p$: the occupation time lies in $[N|I|p,\,N|I|p+r]$, while $T|I|=N|I|p+r|I|$; the difference lies in $[-r|I|,\,r(1-|I|)]$, of absolute value at most $r<p$.
\end{proof}

\begin{lemma}[Syndetic returns to open sets]\label{lem:syndetic}
Let $\omega\in\R^d$ satisfy $k\cdot\omega\neq0$ for all $k\in\Z^d\setminus\{0\}$, and let $U\subset\T^d$ be open and nonempty. Then there is $L=L(U,\omega)<\infty$ such that for every $\theta_0\in\T^d$ and every $t_0\in\R$,
\[
\Leb\{t\in[t_0,t_0+L] : \theta_0+t\omega\in U\}>0.
\]
Equivalently: every time interval of length $L$ contains a set of visit times to $U$ of positive Lebesgue measure --- in particular a nonempty one, so that no interval of length $L$ is free of visits. (The proof gives the quantitative form: that set has measure at least $cL$, with $c=c(U,\omega)>0$.)
\end{lemma}

\begin{proof}
Pick $\psi\in C(\T^d)$ with $0\le\psi\le\mathbf 1_U$ and $\int_{\T^d}\psi\,d\lambda_d=:2c>0$. By unique ergodicity of the flow \eqref{eq:kron}, $\frac1L\int_0^L\psi(\theta+t\omega)\,dt\to\int\psi\,d\lambda_d$ as $L\to\infty$, uniformly in the initial phase $\theta\in\T^d$ \cite[Ch.~1]{CFS}; uniformity in the window is then automatic, because the window $[t_0,t_0+L]$ for the initial phase $\theta_0$ coincides with the window $[0,L]$ for the initial phase $\theta_0+t_0\omega$. Choose $L$ so large that $\frac1L\int_{t_0}^{t_0+L}\psi(\theta_0+t\omega)\,dt\ge c$ for all $\theta_0\in\T^d$, $t_0\in\R$. Since $\psi\le\mathbf 1_U$, the set of times $t\in[t_0,t_0+L]$ with $\theta_0+t\omega\in U$ has Lebesgue measure at least $cL>0$.
\end{proof}

\section{The quantitative equidistribution theorem}\label{sec:main}

For $\varphi\in C(\Sph^d)$, $\theta_0\in\T^d$, $t_0\in\R$ and $T>0$ define the window average
\[
A_{t_0,T}(\varphi):=\frac1T\int_{t_0}^{t_0+T}\varphi\bigl(\Phi_d(\theta_0+t\omega)\bigr)\,dt.
\]

\begin{theorem}[Quantitative window-uniform equidistribution]\label{thm:main}
Let $d\ge2$, $\gamma\in(0,1]$, $\tau\ge d-1$, $\omega\in\mathcal D(\gamma,\tau)$, and $0<s\le1$. Set
\[
\vartheta=\frac{s}{(1+s)(\tau+d)}.
\]
There is a constant $C=C(d,\gamma,\tau,s)<\infty$ such that for every $\varphi\in C^s(\Sph^d)$, every $\theta_0\in\T^d$, every $t_0\in\R$ and every $T>0$,
\[
\bigl|A_{t_0,T}(\varphi)-\langle\varphi\rangle\bigr|\le C\,\|\varphi\|_{C^s(\Sph^d)}\,T^{-\vartheta}.
\]
\end{theorem}

\begin{proof}
Throughout, $C$ denotes a constant depending only on $(d,\gamma,\tau,s)$, and we may assume $T\ge T_0(d,\gamma,\tau,s)$ large, the complementary range being absorbed by enlarging $C$ (the left-hand side never exceeds $2\|\varphi\|_\infty$). Put $F=\varphi\circ\Phi_d\colon\T^d\to\R$, so that $\int F\,d\lambda_d=\langle\varphi\rangle$ by Lemma~\ref{lem:push}, and fix two parameters
\[
\eta\in(0,\tfrac18],\qquad K\in\N,
\]
to be chosen at the end.

\emph{Localization.} Let $\xi_\eta\colon\T^1\to[0,1]$ be $1$-periodic, with $\xi_\eta=0$ on $[0,\eta]\cup[1-\eta,1]$, $\xi_\eta=1$ on $[2\eta,1-2\eta]$, linear in between, so that $\xi_\eta$ is Lipschitz with constant $\le2/\eta$ (note that $\xi_\eta(0)=\xi_\eta(1)=0$, so $\xi_\eta$ is indeed continuous, and $(2/\eta)$-Lipschitz, with respect to $\dist_{\T}$). Define
\[
X(u)=\prod_{m=2}^{d}\xi_\eta(u_m),\qquad G=F\,X.
\]
Then $X$ is supported in $\T^d_\eta$ and $X=1$ off the set $W_\eta=\{u : \exists\,m\ge2,\ u_m\in[0,2\eta)\cup(1-2\eta,1]\}$. Observe also that $G\in C(\T^d)$: the factor $X$ vanishes on a neighbourhood of the polar cut $P_d$, which by Remark~\ref{rem:cut} carries all the discontinuities of $F$, and $F$ is continuous elsewhere. Decompose
\begin{equation}\label{eq:decomp}
\begin{aligned}
A_{t_0,T}(\varphi)-\langle\varphi\rangle
&=\underbrace{\Bigl[\frac1T\int_{t_0}^{t_0+T}G(\theta(t))\,dt-\widehat G(0)\Bigr]}_{E_1}\\
&\quad+\underbrace{\frac1T\int_{t_0}^{t_0+T}\bigl(F(1-X)\bigr)(\theta(t))\,dt}_{E_2}
-\underbrace{\int_{\T^d}F(1-X)\,d\lambda_d}_{E_3}.
\end{aligned}
\end{equation}

\emph{Bounds for $E_2$ and $E_3$.} Since $0\le1-X\le1$ vanishes off $W_\eta$ and $\lambda_d(W_\eta)\le4(d-1)\eta$,
\[
|E_3|\le\|\varphi\|_\infty\,4(d-1)\eta.
\]
For $E_2$, the integrand is nonzero only when some coordinate $\theta_m(t)=\theta_{0,m}+\omega_mt$, $m\ge2$, lies in an arc of length $4\eta$. Each such coordinate is a circle rotation of speed $|\omega_m|\ge\gamma$ (take $k=e_m$ in Definition~\ref{def:dioph}), so Lemma~\ref{lem:occupation} gives, for each $m$, a time fraction at most $4\eta+\frac1{\gamma T}$. Hence
\[
|E_2|\le\|\varphi\|_\infty\,(d-1)\Bigl(4\eta+\frac1{\gamma T}\Bigr).
\]

\emph{H\"older seminorm of the localized function.} We claim that, with respect to the torus metric $\dist_{\T^d}$,
\begin{equation}\label{eq:Gseminorm}
[G]_s\le C\,\|\varphi\|_{C^s}\,\eta^{-s}.
\end{equation}
Let $u,u'\in\T^d$ and $\delta=\dist_{\T^d}(u,u')$; read all coordinates through their representatives in $[0,1)$. If $G(u)=G(u')=0$ there is nothing to prove. Otherwise, exchanging $u$ and $u'$ if necessary, we may assume $G(u)\neq0$; then $u_m\in(\eta,1-\eta)$ for all $2\le m\le d$. We distinguish two cases, according to whether the torus distance is realized across the cut or not.

\emph{Case 1: some coordinate wraps.} Suppose that for some $m\ge2$ the geodesic realizing $\dist_{\T}(u_m,u'_m)$ crosses the cut $\{u_m=0\}$. Then
\[
\delta\ge\dist_{\T}(u_m,u'_m)=\dist_{\T}(u_m,0)+\dist_{\T}(0,u'_m)\ge\min(u_m,1-u_m)\ge\eta,
\]
and the trivial bound gives
\[
|G(u)-G(u')|\le2\|\varphi\|_\infty\le2\|\varphi\|_\infty\,\eta^{-s}\delta^s.
\]

\emph{Case 2: no coordinate of index $m\ge2$ wraps.} Then $\delta=\dc(u,u')$, so Lemma~\ref{lem:reg} applies with $\delta$ on the right-hand side. If $u'\notin\T^d_\eta$, then $X(u')=0$ and, since $|\xi_\eta(a)-\xi_\eta(b)|\le\min(1,\frac2\eta\dist_{\T}(a,b))\le(\frac2\eta)^s\dist_{\T}(a,b)^s$ and the factors of $X$ lie in $[0,1]$, telescoping gives $|X(u)-X(u')|\le(d-1)(\frac2\eta)^s\delta^s$, whence $|G(u)-G(u')|=|F(u)|\,|X(u)-X(u')|\le\|\varphi\|_\infty(d-1)(\frac2\eta)^s\delta^s$. If instead $u'\in\T^d_\eta$, then by Lemma~\ref{lem:reg}(b)
\begin{align*}
|F(u)-F(u')|
&\le[\varphi]_s\,\bigl|\Phi_d(u)-\Phi_d(u')\bigr|^s\\
&\le[\varphi]_s\,C_d^s\,\eta^{s(\frac1d-1)}\,\delta^s
\le C\,[\varphi]_s\,\eta^{-s}\,\delta^s.
\end{align*}
Combining this with
\[
|G(u)-G(u')|\le|X(u)|\,|F(u)-F(u')|+|F(u')|\,|X(u)-X(u')|
\]
proves \eqref{eq:Gseminorm}.

It is worth recording where the loss $\eta^{-s}$ in \eqref{eq:Gseminorm} originates, since this determines what can and cannot be improved. It is carried by the \emph{scale} of the cutoff --- by the Lipschitz constant $2/\eta$ of $\xi_\eta$ and by the trivial bound of Case~1 --- and not by the geometric exponent supplied by Lemma~\ref{lem:reg}(b), whose contribution is only $\eta^{s(\frac1d-1)}$, strictly smaller than $\eta^{-s}$ because $\eta<1$ and $\frac1d>0$. Of course the cutoff is present only because $\Phi_d$ has a cut, so the two are not independent; the point is the operational one that sharpening the geometric step alone cannot improve \eqref{eq:Gseminorm}, nor the exponent $\vartheta$ obtained below. What a genuinely different argument would have to supply is discussed in Remark~\ref{rem:singlescale}.

\emph{Bound for $E_1$.} Let $P=V_K^{\otimes d}G$ as in Lemma~\ref{lem:vdp} (legitimate since $G\in C^s(\T^d)$ by \eqref{eq:Gseminorm}); then $\widehat P(0)=\widehat G(0)$ and
\begin{align*}
E_1
&=\frac1T\int_{t_0}^{t_0+T}\bigl(G-P\bigr)(\theta(t))\,dt\\
&\quad+\sum_{0<|k|_\infty\le2K+1}\widehat P(k)\,e(k\cdot\theta_0)
\,\frac1T\int_{t_0}^{t_0+T}e\bigl(k\cdot\omega\,t\bigr)\,dt.
\end{align*}
The first term is at most $\|G-P\|_\infty\le C[G]_sK^{-s}$. For the second, combine Lemmas~\ref{lem:osc}, \ref{lem:fourier} and \ref{lem:vdp} with the Diophantine bound $|k\cdot\omega|\ge\gamma|k|_\infty^{-\tau}$:
\begin{align*}
\sum_{0<|k|_\infty\le2K+1}\frac{|\widehat G(k)|}{\pi T\,|k\cdot\omega|}
&\le\frac{[G]_s}{2^{1+s}\pi\gamma T}
  \sum_{j=1}^{2K+1}\ \sum_{|k|_\infty=j}j^{\tau-s}\\
&\le\frac{C\,[G]_s}{\gamma T}\sum_{j=1}^{2K+1}j^{\tau+d-1-s}\\
&\le\frac{C\,[G]_s\,K^{\tau+d-s}}{\gamma T},
\end{align*}
where we used $\#\{k:|k|_\infty=j\}\le2d(2j+1)^{d-1}\le C_d\,j^{d-1}$ and $\tau+d-1-s\ge0$. Hence, by \eqref{eq:Gseminorm},
\[
|E_1|\le C\,\|\varphi\|_{C^s}\,\eta^{-s}\Bigl(K^{-s}+\frac{K^{\tau+d-s}}{\gamma T}\Bigr).
\]

\emph{Conclusion.} Collecting the three bounds,
\[
\bigl|A_{t_0,T}(\varphi)-\langle\varphi\rangle\bigr|\le C\,\|\varphi\|_{C^s}\Bigl[\eta^{-s}K^{-s}+\eta^{-s}\frac{K^{\tau+d-s}}{T}+\eta+\frac1T\Bigr].
\]
Choose $K=\lceil T^{1/(\tau+d)}\rceil$, which balances the two $K$-terms at $\eta^{-s}T^{-s/(\tau+d)}$, and then $\eta=\frac18T^{-s/((1+s)(\tau+d))}$, which balances $\eta^{-s}T^{-s/(\tau+d)}$ against $\eta$. All four terms are then $\le CT^{-\vartheta}$ (the term $1/T$ is of lower order since $\vartheta<1$), which is the assertion.
\end{proof}

\begin{remark}[Window uniformity is phase uniformity]\label{rem:window}
The uniformity in $t_0$ asserted in Theorem~\ref{thm:main} is not an independent statement. Since the flow is a translation,
\[
\frac1T\int_{t_0}^{t_0+T}\varphi\bigl(\Phi_d(\theta_0+t\omega)\bigr)\,dt
=\frac1T\int_{0}^{T}\varphi\bigl(\Phi_d(\theta_0'+t\omega)\bigr)\,dt,
\qquad \theta_0'=\theta_0+t_0\omega,
\]
so a bound at $t_0=0$ that is uniform over initial phases is automatically uniform over windows, and conversely. What the proof delivers is the uniformity in $\theta_0$: in the bound for $E_1$ the initial phase enters only through the unimodular factors $e(k\cdot\theta_0)$, and Lemmas~\ref{lem:osc} and~\ref{lem:occupation} are stated with a window position that does not appear in the estimate. We nevertheless carry $t_0$ explicitly, because Theorem~\ref{thm:spiralblock} applies the estimate on consecutive blocks $I_j=[jL,(j+1)L)$ with one fixed $\theta_0$, and it is convenient not to have to relabel the phase at each block. The same observation is used, in the same way, in the proof of Lemma~\ref{lem:syndetic}.
\end{remark}

\begin{corollary}[Almost every frequency]\label{cor:ae}
Fix $d\ge2$ and $0<s\le1$. For almost every $\omega\in\R^d$ the following holds: for every $\varepsilon>0$ there is $C_{\omega,\varepsilon}$ such that for all $\varphi\in C^s(\Sph^d)$, all $\theta_0$, $t_0$ and all $T\ge1$,
\[
\bigl|A_{t_0,T}(\varphi)-\langle\varphi\rangle\bigr|\le C_{\omega,\varepsilon}\,\|\varphi\|_{C^s}\,T^{-\frac{s}{(1+s)(2d-1)}+\varepsilon}.
\]
\end{corollary}

\begin{proof}
For each $n\ge1$, the Borel--Cantelli statement following Definition~\ref{def:dioph} gives a full-measure set of frequencies $\omega$ for which $\omega\in\mathcal D(\gamma_n,d-1+1/n)$ for some $\gamma_n=\gamma_n(\omega)>0$. Intersecting these sets over $n$ still gives a full-measure set. Fix $\omega$ in that intersection and $\varepsilon>0$. Choose $n$ so large that
\[
\frac{s}{(1+s)(2d-1+1/n)}\ge
\frac{s}{(1+s)(2d-1)}-\varepsilon.
\]
Applying Theorem~\ref{thm:main} with $\tau=d-1+1/n$ and $\gamma=\gamma_n$ yields the claimed estimate for $T\ge1$, since increasing the decay exponent only decreases the right-hand side on that range.
\end{proof}

\begin{corollary}[The continuous variant]\label{cor:tilde}
Under the hypotheses of Theorem~\ref{thm:main}, the same estimate holds for the continuous measure-preserving map $\widetilde\Phi_d$ of Remark~\ref{rem:continuousparam}: for every $\varphi\in C^s(\Sph^d)$, every $\theta_0\in\T^d$, every $t_0\in\R$ and every $T>0$,
\[
\Bigl|\frac1T\int_{t_0}^{t_0+T}\varphi\bigl(\widetilde\Phi_d(\theta_0+t\omega)\bigr)\,dt-\langle\varphi\rangle\Bigr|
\le C\,\|\varphi\|_{C^s(\Sph^d)}\,T^{-\vartheta},
\]
with the same $C=C(d,\gamma,\tau,s)$ up to an absolute factor. This exponent is at least as large as the exponent $\frac{s}{d(\tau+d)}$ that the global $\frac1d$-H\"older regularity of $\widetilde\Phi_d$ yields by itself, since $1+s\le2\le d$; the improvement is strict except at the endpoint $d=2$, $s=1$, where the two exponents coincide.
\end{corollary}

\begin{proof}
Let $Q(u)=(u_1,q(u_2),\dots,q(u_d))$. Because Definition~\ref{def:map} reads the torus coordinates through $[0,1)$ whereas $q(1/2)=1$, the pointwise identity $\widetilde\Phi_d=\Phi_d\circ Q$ may fail on
\[
H:=\{u\in\T^d:\ q(u_m)=1\text{ for some }m\ge2\}.
\]
Off $H$ the identity holds, and $H$ has Haar measure zero. Moreover, along every orbit $u=\theta_0+t\omega$ with $\omega\in\mathcal D(\gamma,\tau)$, the set of times hitting $H$ is discrete because $|\omega_m|\ge\gamma$; hence changing the composition on $H$ affects neither the space means nor any of the time averages. We may therefore use $\varphi\circ\widetilde\Phi_d=F\circ Q$ almost everywhere, with $F=\varphi\circ\Phi_d$. Since $q_\#\lambda_1$ is Lebesgue measure on $[0,1]$ (identified with Haar measure on $\T$ up to its null endpoints), one has $Q_\#\lambda_d=\lambda_d$ in this measure-theoretic sense, so every mean value occurring in \eqref{eq:decomp} is unchanged when $F$ and $X$ are replaced by their compositions with $Q$.

Run the proof of Theorem~\ref{thm:main} with $G$ replaced by $(F\,X)\circ Q$. The product $FX$ vanishes near the endpoint cuts, so it has an unambiguous continuous representative on the interval coordinates. Since $q$ is $2$-Lipschitz from $\T$ to $[0,1]$,
\[
\max\Bigl(\dist_{\T}(u_1,u_1'),\max_{m\ge2}|q(u_m)-q(u_m')|\Bigr)
\le2\dist_{\T^d}(u,u').
\]
The proof of \eqref{eq:Gseminorm} also gives the same H\"older bound for $FX$ when the latitude coordinates are measured by their interval distance (indeed $\dist_{\T^d}\le\dc$), and therefore $[(FX)\circ Q]_s\le2^s C\|\varphi\|_{C^s}\eta^{-s}$. Thus \eqref{eq:Gseminorm} persists with an extra absolute factor. For $E_2$ and $E_3$, note that for each $m\ge2$
\[
\{v\in\T:\ q(v)\in[0,2\eta)\cup(1-2\eta,1]\}=[0,\eta)\cup(1-\eta,1]\ \cup\ (\tfrac12-\eta,\tfrac12+\eta),
\]
a union of two arcs of total length $4\eta$; applying Lemma~\ref{lem:occupation} to each arc separately gives the same bounds up to a factor $2$. The bound for $E_1$ is unchanged, and the choice of $\eta$ and $K$ is the same.
\end{proof}

\begin{remark}[Role of localization for the monotone parametrization]\label{rem:localization}
For the specific monotone parametrization $\Phi_d$, localization plays two distinct roles.

First, it makes the Fourier argument on the torus available for general spherical data: by Remark~\ref{rem:cut}(b)--(c), the composition $F=\varphi\circ\Phi_d$ is, for every $\varphi$ that separates some antipodal pair produced by the cut, discontinuous on $\T^d$; such an $F$ admits no modulus of continuity with respect to $\dist_{\T^d}$, and the approximation scheme of Lemma~\ref{lem:vdp} cannot be applied to it globally. Multiplying by the cutoff $X$ removes a neighbourhood of the cut and restores torus-continuity, indeed torus-H\"older regularity, of the localized composition, at the quantified price \eqref{eq:Gseminorm}.

Second, on the subclass of data for which a cutoff-free argument is available at all the comparison goes \emph{against} localization, and it is worth seeing by how much. Suppose $\varphi$ vanishes on a fixed neighbourhood of the subsphere $\Sigma=\{x_1=x_2=0\}$ of Remark~\ref{rem:cut}(b) (for $d=2$: of the two poles). Unwinding the recursion of Definition~\ref{def:map}, the first two Euclidean coordinates of $\Phi_d(u)$ are $\bigl(\prod_{2\le\ell\le d}\rho_\ell(u_\ell)\bigr)\,(\cos2\pi u_1,\sin2\pi u_1)$, whence
\[
\dist\bigl(\Phi_d(u),\Sigma\bigr)\asymp\prod_{2\le\ell\le d}\rho_\ell(u_\ell),
\]
and since $\rho_\ell\le1$ each factor is bounded below as soon as the product is. Quantitatively, if $\varphi$ vanishes on $\{\dist(x,\Sigma)<\varepsilon_0\}$ then $F=\varphi\circ\Phi_d$ vanishes unless $\rho_\ell(u_\ell)\ge c_d\varepsilon_0$ for every $\ell$, and the estimates $\rho_\ell\asymp_\ell w(u_\ell)^{1/\ell}$ of Step~1 of the proof of Lemma~\ref{lem:reg} turn this into $w(u_\ell)\gtrsim_d\varepsilon_0^{\,\ell}$. There is therefore a scale
\[
\eta_0=\eta_0(\varphi)\asymp_d\varepsilon_0^{\,d}>0,
\]
\emph{independent of $T$}, with $\supp F\subset\T^d_{\eta_0}$; shrinking $\eta_0$ slightly we may also assume that $F$ vanishes on a neighbourhood of the boundary of $\T^d_{\eta_0}$. The polynomial degeneration of $\eta_0$ in $\varepsilon_0$ should be kept in mind when reading the comparison below: the cutoff-free constant is worse than the bare $\varepsilon_0^{-s}$ one might expect.

On such data Lemma~\ref{lem:reg}(b) is available, and not merely Lemma~\ref{lem:reg}(a). We spell out the resulting estimate, since it is the one the comparison rests on. Let $u,u'\in\T^d$ and $\delta=\dist_{\T^d}(u,u')$, and assume $F(u)\neq0$, so that $u\in\T^d_{\eta_0}$.

If some coordinate of index $m\ge2$ wraps, i.e.\ the geodesic realizing $\dist_\T(u_m,u'_m)$ crosses $\{u_m=0\}$, then as in Case~1 of the proof of Theorem~\ref{thm:main} one has $\delta\ge\eta_0$, and the trivial bound gives $|F(u)-F(u')|\le2\|\varphi\|_\infty\le2\|\varphi\|_\infty\eta_0^{-s}\delta^s$.

If no such coordinate wraps, then $\delta=\dc(u,u')$. When $u'\in\T^d_{\eta_0}$, Lemma~\ref{lem:reg}(b) applies to the pair directly and
\[
|F(u)-F(u')|\le[\varphi]_s\,\bigl|\Phi_d(u)-\Phi_d(u')\bigr|^s\le[\varphi]_s\,C_d^{\,s}\,\eta_0^{s(\frac1d-1)}\,\delta^s .
\]
When $u'\notin\T^d_{\eta_0}$, the coordinate path from $u$ to $u'$ leaves $\T^d_{\eta_0}$, and we interpose the first point $u^*$ of that path lying outside $\T^d_{\eta_0}$; by the choice of $\eta_0$, $F(u^*)=0=F(u')$, while the pair $(u,u^*)$ lies in $\overline{\T^d_{\eta_0}}$ and satisfies $\dc(u,u^*)\le\delta$, so the previous display gives $|F(u)-F(u')|=|F(u)-F(u^*)|\le[\varphi]_sC_d^{\,s}\eta_0^{s(\frac1d-1)}\delta^s$.

Since $\eta_0<1$ and $\frac1d-1\ge-1$, we have $\eta_0^{s(\frac1d-1)}\le\eta_0^{-s}$, and combining the two cases yields $F\in C^s(\T^d)$ with respect to $\dist_{\T^d}$ and
\[
[F]_s\le C_d\,\|\varphi\|_{C^s}\,\eta_0^{s(\frac1d-1)}\le C_d\,\|\varphi\|_{C^s}\,\eta_0^{-s}.
\]
Only the weaker of the two bounds is used below; we record the sharper one because it shows, as in the discussion following \eqref{eq:Gseminorm}, that the geometric exponent of Lemma~\ref{lem:reg}(b) is not where the loss sits.
Thus $F$ is H\"older of exponent $s$, not of exponent $s/d$ as Lemma~\ref{lem:reg}(a) alone would give. Running Lemmas~\ref{lem:fourier}, \ref{lem:vdp} and \ref{lem:osc} on $F$ with $K=\lceil T^{1/(\tau+d)}\rceil$ and no cutoff then yields
\[
\bigl|A_{t_0,T}(\varphi)-\langle\varphi\rangle\bigr|
\le C_d\,\|\varphi\|_{C^s}\,\eta_0^{-s}\,T^{-\frac{s}{\tau+d}},
\]
an exponent larger than $\vartheta=\frac{s}{(1+s)(\tau+d)}$ by exactly the factor $1+s$, for every $d\ge2$ and every $s\in(0,1]$.

This describes the cost of the localization rather than a defect of it. The constant above degenerates as $\eta_0\to0$ at the rate $\eta_0^{-s}$, and for a general $\varphi\in C^s(\Sph^d)$ no positive $\eta_0$ is available at all: the cutoff manufactures one, and $\vartheta$ is what results from letting the manufactured scale depend on $T$ and balancing $\eta^{-s}$ against the time the flow spends in the discarded collar (Lemma~\ref{lem:occupation}). The factor $1+s$ is the price of that balancing; it is the same in every dimension, although $\vartheta$ and $\frac{s}{\tau+d}$ both decrease with $d$. Whether the factor can be recovered is Problem~\ref{prob:sharp}.
\end{remark}

\begin{remark}[Subtracting the jump: an explicit alternative when $d=2$]\label{rem:jump}
For $d=2$ the discontinuity of $F=\varphi\circ\Phi_2$ across the cut can be removed by an elementary explicit correction. Recording this makes precise how much the localization of Theorem~\ref{thm:main} gains over the cheaper route, and where the cheaper route stops.

Write $N=(0,0,1)$ and $S=(0,0,-1)$. Since $\rho_2(u_2)\to0$ both as $u_2\to0^+$ and as $u_2\to1^-$, one has $\Phi_2(u_1,u_2)\to S$ and $\Phi_2(u_1,u_2)\to N$ respectively, \emph{uniformly in $u_1$}: the limits do not depend on the longitude. The jump of $F$ across $\{u_2=0\}$ is therefore the constant
\[
c_\varphi:=\varphi(N)-\varphi(S),\qquad |c_\varphi|\le2^s[\varphi]_s .
\]
Reading $u_2$ through its representative in $[0,1)$, set $\widetilde F(u_1,u_2):=F(u_1,u_2)-c_\varphi u_2$. Then $\widetilde F(u_1,0)=\varphi(S)=\widetilde F(u_1,1^-)$ for every $u_1$, so $\widetilde F$ is continuous on $\T^2$, and moreover
\[
[\widetilde F]_{s/2}\le C\,[\varphi]_s\qquad\text{with respect to }\dist_{\T^2}.
\]
Indeed, for a pair whose geodesics avoid the cut one has $\dist_{\T^2}=\dc=:\delta$ and it suffices to combine Lemma~\ref{lem:reg}(a) with $|c_\varphi||u_2-u_2'|\le|c_\varphi|\delta\le|c_\varphi|\delta^{s/2}$, valid since $\delta\le1$; while for a pair $u=(u_1,\varepsilon)$, $u'=(u_1',1-\varepsilon')$ whose second coordinate wraps one routes through the cut,
\begin{align*}
|\widetilde F(u)-\widetilde F(u')|
&\le|\widetilde F(u_1,\varepsilon)-\widetilde F(u_1,0)|
+|\widetilde F(u_1,0)-\widetilde F(u_1',1^-)|\\
&\quad+|\widetilde F(u_1',1^-)-\widetilde F(u_1',1-\varepsilon')|,
\end{align*}
where the middle term vanishes because both values equal $\varphi(S)$, and where $\varepsilon,\varepsilon'\le\dist_{\T}(u_2,u_2')\le\delta$.

Since $t\mapsto\theta_{0,2}+\omega_2t$ is a circle rotation of speed $|\omega_2|\ge\gamma$, the sawtooth term is averaged exactly: splitting the window into full periods and a remainder, as in Lemma~\ref{lem:occupation},
\[
\Bigl|\frac1T\int_{t_0}^{t_0+T}\{\theta_{0,2}+\omega_2t\}\,dt-\frac12\Bigr|\le\frac1{\gamma T},
\]
and $\int_{\T^2}F=\int_{\T^2}\widetilde F+\frac12c_\varphi$. Running Lemmas~\ref{lem:fourier}, \ref{lem:vdp} and \ref{lem:osc} on $\widetilde F$ with $K=\lceil T^{1/(\tau+2)}\rceil$ and \emph{no} cutoff therefore gives
\[
\bigl|A_{t_0,T}(\varphi)-\langle\varphi\rangle\bigr|
\le C\,\|\varphi\|_{C^s}\Bigl(T^{-\frac{s}{2(\tau+2)}}+T^{-1}\Bigr).
\]
Comparing exponents, $\frac{s}{2(\tau+2)}\le\vartheta=\frac{s}{(1+s)(\tau+2)}$ because $1+s\le2$, with equality exactly at the Lipschitz endpoint $s=1$. In dimension two the localization thus gains for $s<1$ and gains nothing at $s=1$.

The route does not transfer verbatim to higher dimensions, and it is worth being exact about what does and does not survive. For $d\ge3$ and $2\le m\le d$, the two one-sided limits of $\Phi_d$ along $\{u_m=0\}$ differ by the vector \eqref{eq:jumpvector} of Remark~\ref{rem:cut}(c), namely $2\bigl(\prod_{m<\ell\le d}\rho_\ell(u_\ell)\bigr)e_{m+1}$, which, the functions $\rho_\ell$ being continuous and $1$-periodic on $\T$ (Step~1 of the proof of Lemma~\ref{lem:reg}), is a \emph{continuous} function of the remaining latitudes: the discontinuity of $\Phi_d$ across a single cut is thus no worse than for $d=2$, and it even degenerates where those latitudes approach their own poles. What fails is one step later. The jump of $F=\varphi\circ\Phi_d$ across $\{u_m=0\}$ is
\[
J_m=\varphi\bigl(\Phi_d|_{u_m=1^-}\bigr)-\varphi\bigl(\Phi_d|_{u_m=0^+}\bigr),
\]
and the two arguments involve $z_\ell(u_\ell)$ for $\ell>m$, so $J_m$ is itself discontinuous across the remaining cuts. Subtracting $J_m(\cdot)\,u_m$ therefore restores continuity in the $m$-th latitude only, and the correction has to be applied again to the remaining ones, $d-1$ times in all. We have not carried out this recursion and we do not claim that it must degrade; the point we do make is the comparative one, that localization treats all latitudes simultaneously and returns $\vartheta$ by one argument in every dimension, whereas a recursive subtraction requires a separate argument at each latitude and, for $d=2$ where it can be carried out in closed form, stops at $\frac{s}{2(\tau+2)}$. Whether the recursion yields a better exponent for $d\ge3$ is left open in Problem~\ref{prob:sharp}.
\end{remark}

\begin{lemma}[Hardy--Krause variation of smooth pullbacks]\label{lem:HKpullback}
Identify $\T^d$ with the half-open cube $[0,1)^d$ and use on $[0,1]^d$ the interval extension of $\Phi_d$ from Definition~\ref{def:map}. If $\psi\in C^d(\Sph^d)$, then $H=\psi\circ\Phi_d$ has bounded Hardy--Krause variation on $[0,1]^d$, with the convention anchored at the upper endpoint, and
\[
V_{\mathrm{HK}}(H)\le C_d\,\|\psi\|_{C^d(\Sph^d)}.
\]
More precisely, for every nonempty $I\subset\{1,\dots,d\}$, let $H_I$ be the restriction of $H$ to the face obtained by fixing $u_j=1$ for $j\notin I$. Then
\[
V^{(|I|)}(H_I)\le C_d\,\|\psi\|_{C^{|I|}(\Sph^d)}.
\]
\end{lemma}

\begin{proof}
On the relative interior of the face corresponding to $I$, differentiate once in every free variable. The chain rule produces a finite sum involving derivatives of $\psi$ of order at most $|I|$ and, for every latitude variable $u_\ell$ with $\ell\in I$, $\ell\ge2$, at most one first derivative of $z_\ell$ or $\rho_\ell$. By Step~1 of Lemma~\ref{lem:reg}, while the derivative in $u_1$ is uniformly bounded,
\begin{equation}\label{eq:HKmixed}
\bigl|\partial_I H_I(u_I)\bigr|
\le C_d\,\|\psi\|_{C^{|I|}(\Sph^d)}
\prod_{\substack{\ell\in I\\ \ell\ge2}}w(u_\ell)^{\frac1\ell-1}
\qquad\text{for a.e. }u_I\in(0,1)^{|I|}.
\end{equation}
Every exponent $\frac1\ell-1$ is strictly greater than $-1$, hence the right-hand side belongs to $L^1((0,1)^{|I|})$.

It remains to justify that this integrability controls Vitali variation despite the endpoint singularities. Fix a rectangular partition $\mathcal P$ of the face. For $\delta>0$, move every partition coordinate equal to $0$ or $1$ to $\delta$ or $1-\delta$, respectively, leaving the interior coordinates unchanged; denote the resulting interior partition by $\mathcal P_\delta$. On every rectangle $R$ of $\mathcal P_\delta$, the function $H_I$ is $C^{|I|}$ in a neighbourhood of $R$, so repeated application of the one-dimensional fundamental theorem of calculus gives
\[
\Delta_R H_I=\int_R\partial_I H_I(u_I)\,du_I,
\]
where $\Delta_R$ denotes the alternating rectangular increment. Therefore
\[
\sum_{R\in\mathcal P_\delta}|\Delta_R H_I|
\le\int_{(0,1)^{|I|}}|\partial_I H_I(u_I)|\,du_I
\le C_d\,\|\psi\|_{C^{|I|}(\Sph^d)}.
\]
The interval extension of $\Phi_d$, hence $H_I$, is continuous on the closed face. Letting $\delta\downarrow0$ makes every rectangular increment for $\mathcal P_\delta$ converge to the corresponding increment for $\mathcal P$. Thus the same bound holds for $\mathcal P$. Taking the supremum over all rectangular partitions yields the asserted estimate for $V^{(|I|)}(H_I)$. Summing over all nonempty $I$ is exactly the Hardy--Krause variation anchored at $1$ and proves the result.
\end{proof}

\begin{remark}[Mollification and the Koksma--Hlawka route]\label{rem:HK}
The discrepancy route is not closed off by the low regularity of $\varphi$: it can always be entered by mollifying first. Since this is the natural objection to the argument of Theorem~\ref{thm:main}, we carry it out and compare exponents. The outcome is that neither route dominates, and the threshold is explicit.

\emph{(a) Discrepancy of the flow.} Let $\omega\in\mathcal D(\gamma,\tau)$ and let $D_T=D_T(\theta_0,\omega)$ be the extreme discrepancy of the flow segment $\{\theta_0+t\omega:0\le t\le T\}$ with respect to $\lambda_d$. The continuous form of the Erd\H os--Tur\'an--Koksma inequality \cite[Ch.~2]{KuipersNiederreiter}, \cite[Ch.~1]{DrmotaTichy}, combined with Lemma~\ref{lem:osc} and the Diophantine bound, gives for every $K\in\N$
\[
D_T\le C_d\Bigl(\frac1K+\frac1{\gamma T}\sum_{0<|k|_\infty\le K}\frac{|k|_\infty^{\tau}}{r(k)}\Bigr),
\qquad r(k)=\prod_{m=1}^{d}\max(1,|k_m|).
\]
Since $\sum_{0<|k|_\infty\le K}r(k)^{-1}=O\bigl((\log K)^{d}\bigr)$, the sum is $O\bigl(K^{\tau}(\log K)^{d}\bigr)$, and optimizing in $K$ yields
\begin{equation}\label{eq:discflow}
D_T\le C(d,\gamma,\tau,\varepsilon)\,T^{-\frac1{\tau+1}+\varepsilon}\qquad(\varepsilon>0),
\end{equation}
uniformly in $\theta_0$ and, by Remark~\ref{rem:window}, in the position of the window.

\emph{(b) Finite variation after mollification.} We use Hardy--Krause variation on the fundamental cube $[0,1]^d$, anchored at the upper endpoint. Lemma~\ref{lem:HKpullback} gives the precise statement needed below:
\[
V_{\mathrm{HK}}\bigl(\psi\circ\Phi_d\bigr)\le C_d\,\|\psi\|_{C^{d}(\Sph^d)}
\qquad(\psi\in C^d(\Sph^d)).
\]
In particular, the endpoint singularities of the inverse-CDF coordinates are integrable in every mixed derivative entering the Hardy--Krause face decomposition, and the proof of Lemma~\ref{lem:HKpullback} justifies the passage from those mixed derivatives to Vitali variation rather than assuming absolute continuity at the polar endpoints. For clarity about the form of Koksma--Hlawka used here, let
\[
\mu_T:=\frac1T\int_0^T\delta_{\theta_0+t\omega}\,dt,
\qquad
D_T^*:=\sup_{x\in[0,1]^d}\bigl|\mu_T([0,x))-\lambda_d([0,x))\bigr|.
\]
The measure version of the Koksma--Hlawka inequality gives, for every function $H$ of bounded Hardy--Krause variation (with the matching endpoint convention),
\begin{equation}\label{eq:KHmeasure}
\left|\int_{[0,1]^d}H\,d\mu_T-\int_{[0,1]^d}H\,d\lambda_d\right|
\le V_{\mathrm{HK}}(H)\,D_T^*.
\end{equation}
For completeness, here is the measure-theoretic justification in the convention used above. Put $\nu=\mu_T-\lambda_d$, so that $\nu([0,1]^d)=0$, and for every nonempty $I\subset\{1,\dots,d\}$ write
\[
\Delta_I(x_I):=\nu\bigl([0,x_I)\times[0,1]^{I^c}\bigr).
\]
Each such box is an anchored box (take the remaining coordinates equal to $1$), hence $|\Delta_I(x_I)|\le D_T^*$. The multivariate Stieltjes integration-by-parts formula for Hardy--Krause functions, applied face by face with the same endpoint anchoring as in the definition of $V_{\mathrm{HK}}$, expresses $\int H\,d\nu$ as a finite signed sum of integrals of the functions $\Delta_I$ against the variation measures of the corresponding face restrictions of $H$. Taking absolute values therefore gives
\[
\left|\int_{[0,1]^d}H\,d\nu\right|
\le D_T^*\sum_{\emptyset\ne I\subset\{1,\dots,d\}}
 V^{(|I|)}\!\left(H|_{\text{face }I}\right)
=V_{\mathrm{HK}}(H)D_T^*,
\]
which is \eqref{eq:KHmeasure}. Thus the argument uses only the signed measure $\mu_T-\lambda_d$ and its distribution function on anchored boxes; empirical measures are not essential. Moreover $D_T^*\le D_T$, because every anchored box $[0,x)$ is among the axis-parallel boxes entering the extreme discrepancy. Therefore \eqref{eq:KHmeasure}, applied to $H=\psi\circ\Phi_d$, yields an integration error bounded by $V_{\mathrm{HK}}(\psi\circ\Phi_d)D_T$. This is why we do not assert that $F$ has infinite variation: the obstruction of Section~\ref{sec:intro} is one of low differentiability for the criterion used here, and mollification removes it at a price.

\emph{(c) The resulting exponent.} Let $\varphi\in C^s(\Sph^d)$ and let $\varphi_\epsilon$ be its mollification at scale $\epsilon\in(0,1)$ on $\Sph^d$, so that
\[
\|\varphi-\varphi_\epsilon\|_\infty\le C_d\,[\varphi]_s\,\epsilon^{s},\qquad
\|\varphi_\epsilon\|_{C^{d}(\Sph^d)}\le C_d\,\|\varphi\|_{C^s}\,\epsilon^{s-d}.
\]
Combining the trivial bound for $\varphi-\varphi_\epsilon$ with Koksma--Hlawka for $\varphi_\epsilon$ and \eqref{eq:discflow},
\[
\bigl|A_{t_0,T}(\varphi)-\langle\varphi\rangle\bigr|
\le C\,\|\varphi\|_{C^s}\bigl(\epsilon^{s}+\epsilon^{s-d}\,T^{-\frac1{\tau+1}+\varepsilon}\bigr),
\]
and the choice $\epsilon^{d}=T^{-\frac1{\tau+1}+\varepsilon}$ gives
\begin{equation}\label{eq:HKrate}
\bigl|A_{t_0,T}(\varphi)-\langle\varphi\rangle\bigr|\le C\,\|\varphi\|_{C^s}\,T^{-\vartheta_{\mathrm{HK}}+\varepsilon'},
\qquad
\vartheta_{\mathrm{HK}}=\frac{s}{d(\tau+1)} .
\end{equation}

\emph{(d) Comparison.} A direct computation of the two denominators shows
\begin{equation}\label{eq:threshold}
\vartheta>\vartheta_{\mathrm{HK}}
\iff (1+s)(\tau+d)<d(\tau+1)
\iff \tau\,(d-1-s)>d\,s .
\end{equation}
Three consequences deserve to be stated plainly. First, at the Lipschitz endpoint $s=1$ in dimension $d=2$ the left-hand side of \eqref{eq:threshold} vanishes, so \eqref{eq:HKrate} is strictly better than Theorem~\ref{thm:main}: for $\tau=1$ it gives $\frac14$ against $\vartheta=\frac16$. Second, for every $d\ge2$ and $\tau\ge d-1$ the condition \eqref{eq:threshold} holds for all sufficiently small $s$, and it holds for all $s\in(0,1]$ as soon as $\tau(d-2)>d$; on rough data, and in high dimension, Theorem~\ref{thm:main} is the stronger statement. Third, the two bounds differ in kind and not only in exponent: the localization uses nothing beyond the $C^s$ norm that appears in its conclusion, its constant depends only on $(d,\gamma,\tau,s)$, and it loses no factor $T^{\varepsilon}$, whereas \eqref{eq:HKrate} passes through $C^{d}$ data and through the logarithmic losses in \eqref{eq:discflow}. We therefore regard the two as complementary. Which of them is closer to the truth is part of Problem~\ref{prob:sharp}.
\end{remark}

\begin{lemma}[Cutoff-free Fourier decay at the Lipschitz endpoint]\label{lem:fourierpullback}
Let $d\ge2$ and let $\varphi\in C^1(\Sph^d)$ in the Lipschitz sense of the Notation. Put $F=\varphi\circ\Phi_d$ on the fundamental cube $[0,1)^d$. Then, for every $k\in\Z^d\setminus\{0\}$,
\[
|\widehat F(k)|\le C_d\,\|\varphi\|_{C^1(\Sph^d)}\,|k|_\infty^{-1}.
\]
No cutoff near the polar set is required for this coefficientwise estimate.
\end{lemma}

\begin{proof}
Fix $m\in\{2,\dots,d\}$ and fix the remaining coordinates in the interior of their interval representatives. By Step~1 of Lemma~\ref{lem:reg}, the functions $z_m$ and $\rho_m$ are $C^1$ on $(0,1)$, extend continuously to $[0,1]$, and their derivatives are integrable there. Applying the fundamental theorem of calculus on $[\epsilon,1-\epsilon]$ and then letting $\epsilon\downarrow0$ shows that these endpoint extensions are absolutely continuous on $[0,1]$. The recursive formula for $\Phi_d$ therefore implies that the section
\[
u\longmapsto\Phi_d(u_1,\dots,u_{m-1},u,u_{m+1},\dots,u_d)
\]
is absolutely continuous on $[0,1]$; moreover, for a.e. $u_m\in(0,1)$,
\[
|\partial_{u_m}\Phi_d(u)|\le C_d\,w(u_m)^{\frac1d-1},
\]
and the right-hand side is integrable. Since a Lipschitz function composed with an absolutely continuous curve is absolutely continuous, the corresponding section of $F$ is absolutely continuous, with
\[
|\partial_{u_m}F(u)|\le [\varphi]_1\,|\partial_{u_m}\Phi_d(u)|
\quad\text{a.e.}
\]
Consequently its derivative has $L^1(0,1)$ norm at most $C_d[\varphi]_1$, uniformly in the frozen coordinates. Its endpoint values are understood through the one-sided interval extension of Definition~\ref{def:map}, and their difference is bounded by $2\|\varphi\|_\infty$.

If $k_m\ne0$, one-dimensional integration by parts on $[\epsilon,1-\epsilon]$, followed by $\epsilon\downarrow0$, is therefore justified by absolute convergence and yields, for almost every choice of the remaining coordinates,
\[
\left|\int_0^1F(u)e(-k_mu_m)\,du_m\right|
\le \frac{C_d\,\|\varphi\|_{C^1}}{|k_m|}.
\]
The same estimate holds for $m=1$: the $u_1$-section is periodic and Lipschitz, hence absolutely continuous, and its boundary contribution vanishes. Integrating the preceding one-dimensional estimate over the remaining coordinates by Fubini gives
\[
|\widehat F(k)|\le \frac{C_d\,\|\varphi\|_{C^1}}{|k_m|}
\qquad\text{whenever }k_m\ne0.
\]
Choosing $m$ with $|k_m|=|k|_\infty$ proves the claim.
\end{proof}

\begin{remark}[Why a single cutoff scale, and what a sharper argument would require]\label{rem:singlescale}
Two refinements of the proof of Theorem~\ref{thm:main} suggest themselves, and it is useful to say why the first does not help and what the second would need.

(a) \emph{Dyadic decomposition of the polar collar.} Instead of the single cutoff $X$ at scale $\eta$, one may take a smooth partition subordinate to the shells $\{2^{-j-1}\le\min_{m\ge2}\min(u_m,1-u_m)<2^{-j}\}$, $j\le J$, with $2^{-J}$ the innermost scale. The piece of $F$ carried by the $j$-th shell has H\"older seminorm $\asymp\|\varphi\|_{C^s}2^{js}$ and so contributes $C\|\varphi\|_{C^s}2^{js}\bigl(K^{-s}+K^{\tau+d-s}/T\bigr)$ to the analogue of $E_1$, while the uncut innermost collar contributes $\asymp2^{-J}$ to the analogues of $E_2$ and $E_3$. Since the geometric sum $\sum_{j\le J}2^{js}$ is dominated by its last term, the total reproduces exactly the balance already optimized above, with $2^{-J}$ playing the role of $\eta$. The reason is that $F$ is not small near the polar set: $\|\varphi\|_\infty$ appears on every shell. A multiscale gain would require decay of $\varphi$ towards the subsphere $\Sigma$ of Remark~\ref{rem:cut}(b), which is not assumed here.

(b) \emph{Monotonicity in the latitude variables, and why a cutoff-free argument is not immediate.} Lemma~\ref{lem:fourierpullback} makes the coefficientwise gain precise at the Lipschitz endpoint:
\[
|\widehat F(k)|\le C_d\|\varphi\|_{C^1}|k|_\infty^{-1}
\qquad(k\ne0),
\]
with no cutoff and no factor $\eta^{-s}$. In particular, the one-dimensional integration by parts is legitimate despite the polar endpoint singularities: the relevant sections are absolutely continuous with integrable derivative, and for latitude variables the boundary term is exactly the finite jump across the cut. Thus frequencies with some vanishing coordinates are not the obstruction, since one always integrates in a coordinate $m$ satisfying $|k_m|=|k|_\infty$.

The obstruction lies elsewhere, and it is instructive. The series
\[
\sum_{k\neq0}\frac{|\widehat F(k)|}{|k\cdot\omega|}
\ \lesssim\ \sum_{j\ge1}j^{d-1}\cdot j^{-1}\cdot j^{\tau}
=\sum_{j\ge1}j^{\tau+d-2}
\]
diverges, since $\tau\ge d-1\ge1$; the frequency sum must therefore be truncated, and the truncation error has to be controlled pointwise along a single orbit, that is, in the uniform norm --- which is precisely what the discontinuity of $F$ forbids, since $\|F-V_K^{\otimes d}F\|_\infty$ does not tend to $0$. A cutoff-free argument would consequently have to replace the sup-norm approximation of Lemma~\ref{lem:vdp} by control of the tail on a set of times, i.e.\ by occupation-time information of exactly the kind that the cutoff already supplies.

\emph{Heuristic only.} The following observation is not used in any result of the paper. A possible cutoff-free route would require a quantitative estimate for Fej\'er convergence away from the jump set. If one could prove, for Lipschitz $\varphi$, an error $O\bigl(\|\varphi\|_{C^1}(K\varrho)^{-1}\log K\bigr)$ outside a collar of width $\varrho$, then Lemma~\ref{lem:occupation} would control the time spent inside the collar by $O(\varrho+(\gamma T)^{-1})$. Optimizing this hypothetical bound in $\varrho$ would give a truncation error $O\bigl((K^{-1}\log K)^{1/2}\bigr)$; combining it with the divisor sum $O(K^{\tau+d-1}/T)$ suggests, after choosing $K\asymp T^{1/(\tau+d-\frac12)}$, the candidate exponent
\[
\frac1{2\tau+2d-1}\qquad\text{at }s=1,
\]
in place of $\vartheta|_{s=1}=1/[2(\tau+d)]$. We make no claim that the required Fej\'er estimate holds in the form just stated; the computation is included only as motivation for Problem~\ref{prob:sharp}.
\end{remark}

\begin{remark}[The circle case]\label{rem:circle}
For $d=1$ the orbit $t\mapsto\Phi_1(\theta_0+t\omega)$ is periodic with period $|\omega|^{-1}$. Splitting an arbitrary window into complete periods and a remainder gives directly
\[
|A_{t_0,T}(\varphi)-\langle\varphi\rangle|\le \frac{2\|\varphi\|_\infty}{|\omega|T}
\]
for $\omega\neq0$ (with the obvious harmless modification when $T<|\omega|^{-1}$). Thus the small-divisor and polar-localization phenomena studied here are genuinely higher-dimensional.
\end{remark}

\begin{remark}[Continuous data]\label{rem:continuous}
For merely continuous $\varphi$ convergence persists, without any rate, whenever $k\cdot\omega\neq0$ for all $k\neq0$. Indeed, by Remark~\ref{rem:cut}(b) the composition $F=\varphi\circ\Phi_d$ is bounded and continuous off the closed Haar-null polar cut $P_d$, hence Riemann integrable on $\T^d$, and the classical extension of unique ergodicity to Riemann-integrable observables applies (cf.\ \cite[Ch.~1]{CFS}, \cite{KuipersNiederreiter}). Explicitly, with $X_\eta$ the cutoff from the proof of Theorem~\ref{thm:main}, the functions
\[
h^{\pm}_\eta=F\,X_\eta\pm\|\varphi\|_\infty\,(1-X_\eta)
\]
are continuous on $\T^d$ (the factor $X_\eta$ vanishes near $P_d$), satisfy $h^-_\eta\le F\le h^+_\eta$, and $\int_{\T^d}(h^+_\eta-h^-_\eta)\,d\lambda_d\le8(d-1)\eta\,\|\varphi\|_\infty$; applying the window-uniform convergence for continuous observables to $h^{\pm}_\eta$ and letting $\eta\to0$ gives
\[
A_{t_0,T}(\varphi)\xrightarrow[T\to\infty]{}\langle\varphi\rangle\qquad\text{uniformly in }(\theta_0,t_0),
\]
which is the precise qualitative form of \eqref{eq:qualitative} used in this paper. No quantitative modulus follows from rational independence alone uniformly over the frequency; Theorem~\ref{thm:opt} makes this precise.
\end{remark}

\section{Quantitative spiral means}\label{sec:spirals}

Let $x_0\in\R^{d+1}$, $r_0>0$, and let $f$ be defined on the punctured ball $\{0<|x-x_0|\le r_0\}$. For $0<r\le r_0$ write $f_r(\xi)=f(x_0+r\xi)$, $\xi\in\Sph^d$, and
\[
m(r)=\langle f_r\rangle=\int_{\Sph^d}f(x_0+r\xi)\,d\sigma_d(\xi).
\]
We consider the following hypotheses, for parameters $0<s\le1$, $0<\kappa\le1$ and $\lambda\ge0$:
\begin{enumerate}[label=\textup{(H\arabic*)}, leftmargin=2.6em]
\item\label{H1} \emph{(uniform angular regularity)}\quad $A:=\displaystyle\sup_{0<r\le r_0}\|f_r\|_{C^s(\Sph^d)}<\infty$;
\item\label{H2} \emph{(radial regularity, with degeneration exponent $\lambda$)}
\[
|f_r(\xi)-f_{r'}(\xi)|\le B\,\frac{|r-r'|^{\kappa}}{\min(r,r')^{\lambda}}
\qquad(\xi\in\Sph^d,\ 0<r,r'\le r_0);
\]
\item\label{H3} \emph{(convergence of the spherical means)}\quad there are $m_0\in\R$ and a nondecreasing $\mu\colon(0,r_0]\to[0,\infty)$ with $\mu(0^+)=0$ such that $|m(r)-m_0|\le\mu(r)$ for $0<r\le r_0$.
\end{enumerate}
The case $\lambda=0$ of \ref{H2} is the plain uniform radial H\"older condition; the case $\lambda>0$ allows the radial modulus to degenerate as the centre is approached, at a polynomial rate. Note that no continuity of $f$ at $x_0$ is assumed, only along rays near $x_0$; $m_0$ is the natural spherical mean value at the centre.

Let $g\in C^1([0,\infty))$ with $0<g\le r_0$, $g'\le0$, $g(t)\to0$, and let
\[
\Gamma(t)=x_0+g(t)\,\Phi_d(\theta_0+t\omega)
\]
be the associated shrinking spiral.

\begin{remark}[The trajectory is not a continuous curve]\label{rem:notcurve}
Since $\Phi_d$ is discontinuous across the polar cut, $\Gamma$ is not a continuous curve: at each time at which some coordinate $\theta_m(t)=\theta_{0,m}+\omega_mt$, $m\ge2$, crosses $0$, the point $\Gamma(t)$ jumps between the two one-sided positions described in Remark~\ref{rem:cut}(c), scaled by $g(t)$. These positions are reflections of one another in the corresponding coordinate hyperplane; they are antipodal only when the remaining latitudes are at their equators (and always when $d=2$). Such times form a discrete, hence Lebesgue-null, set, so none of the time averages below is affected, and we use the word ``spiral'' for the resulting Borel trajectory: between two consecutive jumps it is a genuine spiral arc sweeping the latitudes monotonically from pole to pole. The reader who prefers a continuous curve may replace $\Phi_d$ by the variant $\widetilde\Phi_d$ of Remark~\ref{rem:continuousparam}, which traverses each latitude twice and produces a continuous trajectory: by Corollary~\ref{cor:tilde} the input estimate of Theorem~\ref{thm:main} holds verbatim for $\widetilde\Phi_d$, so every result of this section holds for the corresponding continuous spiral with the same constants and exponents, the proofs using no property of $\Phi_d$ beyond Theorem~\ref{thm:main} and the measure preservation of Lemma~\ref{lem:push}.
\end{remark}

\subsection*{The non-degenerate range $\lambda<\kappa$: no block decomposition is needed}

It is worth isolating at the outset a consequence of \ref{H2} which, when $\lambda<\kappa$, makes the whole problem collapse onto Theorem~\ref{thm:main}. The reason is that in that range \ref{H2} forces the family $(f_r)$ to be uniformly Cauchy as $r\to0$, so that the angular profile of $f$ has a limit at the centre in the uniform norm, and the spiral average may be compared with that single limiting profile once and for all.

\begin{lemma}[The limiting angular profile]\label{lem:f0}
Assume \ref{H1} and \ref{H2} with $\lambda<\kappa$, and set $\nu=\kappa-\lambda\in(0,1]$. Then there is $f_0\in C^s(\Sph^d)$ such that
\[
\|f_r-f_0\|_{\infty}\le C_\nu\,B\,r^{\nu}\qquad(0<r\le r_0),
\qquad C_\nu=\frac{2^{\nu}+1}{1-2^{-\nu}},
\]
and moreover $\|f_0\|_{C^s(\Sph^d)}\le A$. Consequently \ref{H3} holds with $m_0=\langle f_0\rangle$ and $\mu(r)=C_\nu Br^{\nu}$.
\end{lemma}

\begin{proof}
Put $\varrho_j=2^{-j}r_0$. By \ref{H2},
\[
\|f_{\varrho_j}-f_{\varrho_{j+1}}\|_\infty\le B\,\frac{(\varrho_{j+1})^{\kappa}}{(\varrho_{j+1})^{\lambda}}=B\,\varrho_{j+1}^{\,\nu},
\]
and $\sum_j\varrho_{j+1}^{\,\nu}<\infty$, so $(f_{\varrho_j})$ converges uniformly to some $f_0$, with
\[
\|f_{\varrho_j}-f_0\|_\infty\le B\sum_{i\ge j}\varrho_{i+1}^{\,\nu}
=B\,\varrho_j^{\,\nu}\,\frac{2^{-\nu}}{1-2^{-\nu}}\le\frac{B\,\varrho_j^{\,\nu}}{1-2^{-\nu}} .
\]
Given $0<r\le r_0$ choose $j\ge0$ with $\varrho_{j+1}<r\le\varrho_j$. Then $|\varrho_j-r|\le\varrho_{j+1}<r$ and $\min(r,\varrho_j)=r$, so \ref{H2} gives $\|f_r-f_{\varrho_j}\|_\infty\le Br^{\kappa-\lambda}=Br^{\nu}$, while $\varrho_j<2r$ gives $\|f_{\varrho_j}-f_0\|_\infty\le\frac{B(2r)^{\nu}}{1-2^{-\nu}}$. Adding the two proves the displayed bound. Since $\|f_r\|_{C^s}\le A$ for every $r$ by \ref{H1}, and the $C^s$ norm is lower semicontinuous under uniform convergence, $\|f_0\|_{C^s}\le A$. Finally $|m(r)-\langle f_0\rangle|=|\langle f_r-f_0\rangle|\le\|f_r-f_0\|_\infty$.
\end{proof}

\begin{theorem}[Quantitative spiral means, non-degenerate case]\label{thm:spiral}
Let $d\ge2$, $\omega\in\mathcal D(\gamma,\tau)$, and assume \ref{H1}--\ref{H2} with $\lambda<\kappa$; write $\nu=\kappa-\lambda$, let $f_0$ be the limiting profile of Lemma~\ref{lem:f0}, set $m_0:=\langle f_0\rangle$, and let $\vartheta$, $C$ be as in Theorem~\ref{thm:main}. Then for every $T>0$,
\[
\Bigl|\frac1T\int_0^T f(\Gamma(t))\,dt-m_0\Bigr|
\le C\,A\,T^{-\vartheta}+C_\nu\,B\,\frac1T\int_0^T g(t)^{\nu}\,dt .
\]
\end{theorem}

\begin{proof}
Write $\xi_t=\Phi_d(\theta_0+t\omega)$, so that $f(\Gamma(t))=f_{g(t)}(\xi_t)$, and split
\[
\frac1T\int_0^Tf(\Gamma(t))\,dt-m_0
=\frac1T\int_0^T\bigl(f_{g(t)}-f_0\bigr)(\xi_t)\,dt
+\Bigl(\frac1T\int_0^Tf_0(\xi_t)\,dt-\langle f_0\rangle\Bigr),
\]
using $m_0=\langle f_0\rangle$ from Lemma~\ref{lem:f0}. The first term is bounded by $\frac1T\int_0^T\|f_{g(t)}-f_0\|_\infty\,dt\le C_\nu B\frac1T\int_0^Tg(t)^{\nu}\,dt$, again by Lemma~\ref{lem:f0}. The second is bounded by $CAT^{-\vartheta}$ by Theorem~\ref{thm:main} applied to the single observable $f_0\in C^s(\Sph^d)$, whose norm is at most $A$.
\end{proof}

\begin{corollary}[Power-law profiles]\label{cor:power}
In the setting of Theorem~\ref{thm:spiral} let $g(t)=r_0(1+t)^{-b}$ with $b>0$. Then for $T\ge1$
\[
\Bigl|\frac1T\int_0^Tf(\Gamma(t))\,dt-m_0\Bigr|
\le C\,A\,T^{-\vartheta}+C'\,B\,r_0^{\nu}\,T^{-\min(b\nu,1)}\log(2+T),
\]
with $C=C(d,\gamma,\tau,s)$ and $C'=C'(\kappa,\lambda,b)$; the logarithm is needed only when $b\nu=1$.
\end{corollary}

\begin{proof}
$\frac1T\int_0^Tg^{\nu}=\frac{r_0^{\nu}}T\int_0^T(1+t)^{-b\nu}\,dt$, which is $O\bigl(r_0^{\nu}T^{-b\nu}\bigr)$ if $b\nu<1$, $O\bigl(r_0^{\nu}T^{-1}\log(2+T)\bigr)$ if $b\nu=1$, and $O\bigl(r_0^{\nu}T^{-1}\bigr)$ if $b\nu>1$. Insert this in Theorem~\ref{thm:spiral}.
\end{proof}

\begin{remark}[Why the radius costs nothing here]\label{rem:noblocks}
Two features of Corollary~\ref{cor:power} should be emphasized, because a block decomposition of the time axis --- freezing the radius on mesoscopic windows $[jL,(j+1)L)$, applying Theorem~\ref{thm:main} on each of them, and optimizing in $L$ --- would produce a strictly weaker statement. First, the angular exponent is the full $\vartheta$ of Theorem~\ref{thm:main}: there is no degradation of the form $\vartheta/(1+\kappa+\vartheta)$, because the radial and the angular errors are decoupled by Lemma~\ref{lem:f0} instead of being balanced against each other through a mesoscopic scale. Second, no relation between $\kappa$, $\lambda$ and $b$ is required: the profile may shrink arbitrarily slowly. The reason is the one identified in Lemma~\ref{lem:f0}: as soon as $\lambda<\kappa$, hypothesis \ref{H2} is strong enough to produce a \emph{single} limiting angular profile $f_0$, uniformly on $\Sph^d$, and the spiral average need only be compared with the average of that one observable. Only the plain window $[0,T]$ is used, so the window uniformity of Theorem~\ref{thm:main} is not needed in this range either. It becomes necessary exactly when no such limiting profile exists, which is the subject of the next subsection.
\end{remark}

\subsection*{The degenerate range $\lambda\ge\kappa$: mesoscopic blocks}

When $\lambda\ge\kappa$ the chaining of Lemma~\ref{lem:f0} diverges and no limiting angular profile need exist. A concrete example on $\Sph^d$: take $f(x_0+r\xi)=a(\xi)\cos\log\frac1r$ with $a\in C^s(\Sph^d)$ and $\langle a\rangle=0$. Then $m\equiv0$, so \ref{H3} holds with $\mu\equiv0$, and \ref{H2} holds with $\kappa=\lambda=1$ because $|\cos\log\frac1r-\cos\log\frac1{r'}|\le|\log\frac r{r'}|\le\frac{|r-r'|}{\min(r,r')}$; but $f_r=a\cos\log\frac1r$ has no uniform limit as $r\to0$ unless $a\equiv0$. In this range the radius genuinely has to be tracked, and it is here that the window uniformity of Theorem~\ref{thm:main} is used.

\begin{theorem}[Quantitative spiral means, degenerate case]\label{thm:spiralblock}
Let $d\ge2$, $\omega\in\mathcal D(\gamma,\tau)$, and assume \ref{H1}, \ref{H2} and \ref{H3}. Let $\vartheta$ be as in Theorem~\ref{thm:main}. There is $C=C(d,\gamma,\tau,s)<\infty$ such that for every $T\ge1$ and every block length $1\le L\le T$, writing $J=\lfloor T/L\rfloor$, $I_j=[jL,(j+1)L)$, $g_j=g(jL)$ and $\Delta_j=g(jL)-g((j+1)L)\ge0$,
\begin{align*}
\Bigl|\frac1T\int_0^T f(\Gamma(t))\,dt-m_0\Bigr|
\le C\Bigl[&A\,L^{-\vartheta}+\bigl(A+\mu(r_0)\bigr)\frac LT\\
&+\frac{BL}T\sum_{j=0}^{J-1}\frac{\Delta_j^{\kappa}}{g\bigl((j+1)L\bigr)^{\lambda}}
+\frac1T\int_0^T \mu\bigl(g(t)\bigr)\,dt\Bigr].
\end{align*}
\end{theorem}

\begin{proof}
Split $[0,T]=\bigcup_{j=0}^{J-1}I_j\cup[JL,T)$. The remainder interval has length $<L$ and $\|f\|_{L^\infty}\le A$ on the range of $\Gamma$ by \ref{H1}, so it contributes at most $AL/T$ to the average; and $|m_0|\le|m(r_0)|+\mu(r_0)\le A+\mu(r_0)$, so replacing $m_0$ by $\frac{JL}Tm_0$ costs at most $(A+\mu(r_0))L/T$.

Fix $j$. For $t\in I_j$ one has $g((j+1)L)\le g(t)\le g_j$ by monotonicity, so $\min(g(t),g_j)=g(t)\ge g((j+1)L)$ and $|g(t)-g_j|\le\Delta_j$; hence \ref{H2} gives
\[
\Bigl|\frac1L\int_{I_j}f(\Gamma(t))\,dt-\frac1L\int_{I_j}f_{g_j}\bigl(\Phi_d(\theta_0+t\omega)\bigr)\,dt\Bigr|
\le B\,\frac{\Delta_j^{\kappa}}{g((j+1)L)^{\lambda}} .
\]
By \ref{H1}, $f_{g_j}\in C^s(\Sph^d)$ with norm at most $A$, so Theorem~\ref{thm:main}, applied on the window $I_j$ --- this is where the window uniformity is used --- gives
\[
\Bigl|\frac1L\int_{I_j}f_{g_j}\bigl(\Phi_d(\theta_0+t\omega)\bigr)\,dt-m(g_j)\Bigr|\le C\,A\,L^{-\vartheta},
\]
while $|m(g_j)-m_0|\le\mu(g_j)$ by \ref{H3}. Averaging over blocks,
\[
\Bigl|\frac1T\int_0^{JL}f(\Gamma)\,dt-\frac{JL}Tm_0\Bigr|
\le\frac LT\sum_{j=0}^{J-1}\Bigl[B\frac{\Delta_j^{\kappa}}{g((j+1)L)^{\lambda}}+CAL^{-\vartheta}+\mu(g_j)\Bigr].
\]
For the last sum, since $g$ is nonincreasing and $\mu$ nondecreasing, $L\mu(g_j)\le\int_{(j-1)L}^{jL}\mu(g(t))\,dt$ for $j\ge1$, while the $j=0$ term is at most $L\mu(r_0)$; hence $\frac LT\sum_j\mu(g_j)\le\frac1T\int_0^T\mu(g)+\frac{\mu(r_0)L}T$. Collecting terms proves the claim.
\end{proof}

\begin{corollary}[Power-law profiles, degenerate case]\label{cor:powerdeg}
In the setting of Theorem~\ref{thm:spiralblock}, assume in addition that $\lambda\ge\kappa$, let $g(t)=r_0(1+t)^{-b}$ with $b>0$, and set
\[
\beta:=\kappa(b+1)-b\lambda=\kappa+b(\kappa-\lambda)\le\kappa\le1.
\]
Assume $\beta>0$. If $0<\beta<1$, then for $T\ge1$,
\begin{align*}
\Bigl|\frac1T\int_0^Tf(\Gamma(t))\,dt-m_0\Bigr|
\le\ &C\bigl(A+\mu(r_0)+B\,r_0^{\kappa-\lambda}\bigr)\,
T^{-\frac{\vartheta\beta}{\kappa+\vartheta}}\\
&+\frac CT\int_0^T\mu\bigl(g(t)\bigr)\,dt.
\end{align*}
If $\beta=1$ (which in the degenerate range forces $\kappa=\lambda=1$), then
\begin{align*}
\Bigl|\frac1T\int_0^Tf(\Gamma(t))\,dt-m_0\Bigr|
\le\ &C\bigl(A+\mu(r_0)+B\bigr)\,
\Bigl(\frac{\log(2+T)}T\Bigr)^{\frac{\vartheta}{1+\vartheta}}\\
&+\frac CT\int_0^T\mu\bigl(g(t)\bigr)\,dt.
\end{align*}
Here $C=C(d,\gamma,\tau,s,\kappa,\lambda,b)$.
\end{corollary}

\begin{proof}
We repeat the block proof of Theorem~\ref{thm:spiralblock}, with one harmless refinement: instead of freezing the radius on the first block $I_0=[0,L)$, discard that block together with the terminal remainder. Their combined contribution, including the corresponding portion of $m_0$, is $O((A+\mu(r_0))L/T)$. On every remaining block $I_j$, $j\ge1$, one has
\[
\Delta_j\le r_0bL(1+jL)^{-b-1}
\]
and, since $L\le jL\le1+jL$,
\[
1+(j+1)L=1+jL+L\le2(1+jL).
\]
Therefore
\[
\frac{\Delta_j^{\kappa}}{g((j+1)L)^{\lambda}}
\le C(b,\kappa,\lambda)\,r_0^{\kappa-\lambda}\,
L^{\kappa}(1+jL)^{-\beta}.
\]
Consequently the radial-freezing contribution of the blocks $j\ge1$ is bounded by
\[
C B r_0^{\kappa-\lambda}\frac{L^{1+\kappa}}T
\sum_{j=1}^{J-1}(1+jL)^{-\beta}.
\]
For completeness, the elementary sum has the three regimes
\[
\sum_{j=1}^{J-1}(1+jL)^{-\beta}\le C_\beta
\begin{cases}
L^{-\beta}J^{1-\beta},&\beta<1,\\
L^{-1}\log(2+J),&\beta=1,\\
L^{-\beta},&\beta>1.
\end{cases}
\]
(The last case cannot occur when $\lambda\ge\kappa$ and $0<\kappa\le1$.) Thus, when $0<\beta<1$, the freezing error is at most
\[
C B r_0^{\kappa-\lambda}L^{\kappa}T^{-\beta}.
\]
Choose $L=T^{\beta/(\kappa+\vartheta)}$. Since $0<\beta\le\kappa<\kappa+\vartheta$, one has $1\le L\le T$, as required in Theorem~\ref{thm:spiralblock}, and
\[
L^{-\vartheta}+L^{\kappa}T^{-\beta}
\le C T^{-\frac{\vartheta\beta}{\kappa+\vartheta}},
\qquad
\frac LT\le C T^{-\frac{\vartheta\beta}{\kappa+\vartheta}}.
\]
The last inequality follows from $\kappa+\vartheta\ge\beta(1+\vartheta)$, which uses $\beta\le\kappa$ and $\beta\le1$. Inserting these bounds into the block proof gives the first assertion.

If $\beta=1$, then $1=\beta\le\kappa\le1$, hence $\kappa=1$, and $1=\kappa+b(\kappa-\lambda)$ with $b>0$ and $\lambda\ge\kappa$ forces $\lambda=1$. The freezing error is then at most
\[
C B\,L\,\frac{\log(2+T)}T.
\]
For $T\ge2$, take
\[
L=\Bigl(\frac{T}{\log(2+T)}\Bigr)^{1/(1+\vartheta)}.
\]
Then $1\le L\le T$, and this choice balances the freezing term with $AL^{-\vartheta}$ and gives
\[
L^{-\vartheta}+L\frac{\log(2+T)}T
\le C\Bigl(\frac{\log(2+T)}T\Bigr)^{\frac{\vartheta}{1+\vartheta}},
\]
while $L/T$ is smaller. The bounded range $1\le T<2$ is absorbed by enlarging the constant, using the trivial bound on the left-hand side. The estimate for the spherical-mean terms is unchanged from Theorem~\ref{thm:spiralblock}, which completes the proof.
\end{proof}

\begin{remark}[Consistency and the cost of the blocks]\label{rem:blockcost}
Theorem~\ref{thm:spiralblock} itself is valid on both sides of the threshold $\lambda=\kappa$, but Corollary~\ref{cor:powerdeg} is stated only in the range where the block decomposition is actually needed. In that degenerate range
\[
\beta=\kappa+b(\kappa-\lambda)\le\kappa\le1,
\]
so the previously tempting summability condition $\beta>1$ can never occur. The correct power-law optimization uses the blocks $j\ge1$ and discards the initial block: for $0<\beta<1$ it gives the exponent $\vartheta\beta/(\kappa+\vartheta)$, while the boundary case $\beta=1$ (necessarily $\kappa=\lambda=1$) carries the logarithmic endpoint factor displayed in Corollary~\ref{cor:powerdeg}. If $\beta\le0$, the general estimate of Theorem~\ref{thm:spiralblock} remains available, but this block argument does not by itself force the radial-freezing error to vanish. By contrast, when $\lambda<\kappa$ Lemma~\ref{lem:f0} produces a single limiting angular profile and Corollary~\ref{cor:power} gives the full angular exponent $\vartheta$ with no compatibility condition on $(\kappa,\lambda,b)$. Thus the block scheme should be viewed as a tool for the genuinely degenerate range, not as a competitor to the limiting-profile argument. We do not know whether the exponents furnished by the block scheme are optimal there.
\end{remark}

\begin{remark}[Purely angular data]\label{rem:angular}
If $f(x_0+r\xi)=\varphi(\xi)$ does not depend on $r$ (i.e.\ \ref{H2} holds with $B=0$, and \ref{H3} with $\mu\equiv0$), then $f(\Gamma(t))=\varphi(\Phi_d(\theta_0+t\omega))$ identically, $f_0=\varphi$, and Theorem~\ref{thm:spiral} reduces to Theorem~\ref{thm:main} with the full rate $T^{-\vartheta}$.
\end{remark}
\section{Rough angular data and weighted spiral means}\label{sec:rough}

So far the data were at least continuous in the angular variable. In this section we move toward the opposite end of the regularity scale in two directions: merely integrable angular data, for which the natural statement is an almost-every-phase theorem, and data with a homogeneous singularity at the center, for which the time average itself must be renormalized. For the latter we exhibit a sharp contrast between two model radial regimes: power-law shrinking transfers convergence, and the rates of Theorem~\ref{thm:main} for H\"older data, whereas exponential shrinking can destroy convergence for every phase.

\begin{theorem}[Integrable angular data]\label{thm:L1}
Let $d\ge1$ and let $\omega\in\R^d$ satisfy $k\cdot\omega\neq0$ for all $k\in\Z^d\setminus\{0\}$. Let $a\in L^1(\Sph^d,\sigma_d)$ (a fixed Borel representative). Then for $\lambda_d$-almost every $\theta_0\in\T^d$,
\[
\lim_{T\to\infty}\frac1T\int_0^Ta\bigl(\Phi_d(\theta_0+t\omega)\bigr)\,dt=\langle a\rangle.
\]
\end{theorem}

\begin{proof}
The flow $S_t\theta=\theta+t\omega$ (the letter $\tau$ being reserved for the Diophantine exponent) preserves $\lambda_d$ and is ergodic: if $h\in L^2(\T^d)$ is invariant, then $\widehat h(k)\bigl(e(k\cdot\omega\,t)-1\bigr)=0$ for all $t$ and all $k$, forcing $\widehat h(k)=0$ for $k\neq0$. The push-forward identity of Lemma~\ref{lem:push} extends from $C(\Sph^d)$ to bounded Borel functions by the monotone class theorem, and to $L^1$ by truncation; hence $F=a\circ\Phi_d\in L^1(\T^d)$ with $\int F\,d\lambda_d=\langle a\rangle$. Moreover $(\theta_0,t)\mapsto F(\theta_0+t\omega)$ is jointly measurable and, by Fubini, $t\mapsto F(\theta_0+t\omega)$ is locally integrable for a.e.\ $\theta_0$. The claim is now Birkhoff's ergodic theorem for measure-preserving flows \cite[Ch.~1]{CFS}.
\end{proof}

\begin{remark}\label{rem:L1sharp}
(a) \emph{The exceptional null set is unavoidable.} Fix $d\ge2$, $\omega$ with $k\cdot\omega\neq0$ for $k\neq0$, and $\theta_0\in\T^d$; we construct $a\in L^1(\Sph^d)$ whose averages diverge at this particular phase. Recall $|\omega_m|>0$ for all $m$ (take $k=e_m$). We first record two facts.

(i) \emph{Recurrence to the bulk with linear time gaps.} The set $U=\intr\T^d_{1/8}$ is open and nonempty, so Lemma~\ref{lem:syndetic} provides $L=L(\omega,d)\ge1$ such that every time window of length $L$ contains a positive measure of times $t$ with $\theta_0+t\omega\in\T^d_{1/8}$. Choosing one visiting time $t_n$ in each window $[n(L+1),\,n(L+1)+L]$, $n\ge1$, we obtain times $t_n\uparrow\infty$ with
\[
n\le t_n\le C_\omega n,\qquad t_{n+1}\ge t_n+1,\qquad \theta_0+t_n\omega\in\T^d_{1/8},
\]
with $C_\omega=2L+1$: indeed $t_n\ge n(L+1)\ge n$, $t_n\le n(L+1)+L\le n(2L+1)$, and $t_{n+1}-t_n\ge(n+1)(L+1)-n(L+1)-L=1$.

(ii) \emph{Bounded speed near the bulk.} While $\theta_0+t\omega\in\T^d_{1/16}$, the curve $t\mapsto\Phi_d(\theta_0+t\omega)$ is differentiable with speed at most $V:=C_d(1/16)^{\frac1d-1}|\omega|$, by the local derivative bounds \eqref{eq:localder} (equivalently, Lemma~\ref{lem:reg}(b), which is a statement about the coordinate cube and hence applies along the trajectory as long as it stays in $\T^d_{1/16}$). Moreover, since each coordinate moves with constant speed $|\omega_m|$, a trajectory starting in $\T^d_{1/8}$ remains in $\T^d_{1/16}$ for a time at least $h_0:=(16\max_m|\omega_m|)^{-1}$.

Now set $r_n:=n2^{-n}$ and let $B_n$ be the \emph{chordal} ball $\{x\in\Sph^d:|x-\Phi_d(\theta_0+t_n\omega)|<r_n\}$; since the geodesic and chordal distances on $\Sph^d$ are comparable, with $|x-y|\le\dist_{\Sph^d}(x,y)\le\frac\pi2|x-y|$, replacing $B_n$ by a geodesic ball changes nothing but the constants. The speed bound of (ii) is stated for the Euclidean, hence chordal, distance, so it applies to $B_n$ directly. By (i)--(ii), for all $n\ge n_0$ large enough that $r_n\le Vh_0$ and $r_n/V<1$ --- possible since $r_n\to0$ --- the curve spends time at least $r_n/V$ inside $B_n$ (it cannot exit a ball of radius $r_n$ in less time at speed $\le V$), and this time is spent in the interval $[t_n,t_n+r_n/V]\subset[t_n,t_n+1]\subset[0,t_{n+1}]$, because $t_{n+1}-t_n\ge1$ by (i). In particular, for $n\le N-1$ these intervals are pairwise disjoint and contained in $[0,t_N]$. Set $a=\sum_n2^n\mathbf 1_{B_n}$: then
\[
\int_{\Sph^d}a\,d\sigma_d=\sum_n2^n\sigma_d(B_n)\le C\sum_n2^nr_n^d=C\sum_nn^d2^{-n(d-1)}<\infty
\]
since $d\ge2$, so $a\in L^1(\Sph^d)$; while, using $t_N\le C_\omega N$,
\begin{align*}
\frac1{t_N}\int_0^{t_N}a\bigl(\Phi_d(\theta_0+t\omega)\bigr)\,dt
&\ge\frac1{C_\omega N}\sum_{n_0\le n\le N-1}2^n\,\frac{r_n}V\\
&=\frac1{C_\omega VN}\sum_{n_0\le n\le N-1}n
\gtrsim N\longrightarrow\infty.
\end{align*}
Thus the averages diverge at this particular phase.

(b) \emph{No uniform rates in this generality.} Without a quantitative regularity class for $a$, no uniform modulus can be expected from an $L^1$ assumption alone. Independently, Section~\ref{sec:opt} shows that even for smooth observables no prescribed modulus is uniform over all rationally independent frequencies.

(c) \emph{Purely angular spiral data.} If $f(x_0+r\xi)=a(\xi)$ is independent of $r$, then $f(\Gamma(t))=a(\Phi_d(\theta_0+t\omega))$ along any shrinking spiral, and Theorem~\ref{thm:L1} shows that the spiral average recovers $\langle a\rangle$ for a.e.\ phase for merely integrable $a$, a regime in which the spherical means $m(r)$ need only be interpreted as integrals.
\end{remark}

We now turn to singular data. The elementary Abelian lemma below transfers convergence, and rates, from plain to weighted time averages.

\begin{theorem}[Abelian transfer to weighted means]\label{thm:abel}
Let $u\in L^1_{\mathrm{loc}}([0,\infty))$ and $A(T)=\frac1T\int_0^Tu(t)\,dt$. Let $w\in C^1([0,\infty))$ be positive and nondecreasing with $W(T)=\int_0^Tw(t)\,dt\to\infty$, and suppose
\[
\Lambda:=\sup_{T\ge T_1}\frac{T\,w(T)}{W(T)}<\infty\qquad\text{for some }T_1>0.
\]
\begin{enumerate}[label=\textup{(\alph*)}, leftmargin=2.2em]
\item If $A(T)\to\ell\in\R$, then $\dfrac1{W(T)}\displaystyle\int_0^Tw(t)\,u(t)\,dt\to\ell$.
\item If $|A(T)-\ell|\le M\,T^{-\vartheta}$ for all $T>0$ and some $\vartheta\in(0,1)$, then
\[
\Bigl|\frac1{W(T)}\int_0^Tw\,u\,dt-\ell\Bigr|\le2\Lambda M\,T^{-\vartheta}\qquad\text{for all }T\ge T_1.
\]
\end{enumerate}
\end{theorem}

\begin{proof}
Write $U(T)=\int_0^Tu=T\ell+T\varepsilon(T)$, $\varepsilon(T)=A(T)-\ell$. Integration by parts and the identity $\int_0^Tw'(t)\,t\,dt=Tw(T)-W(T)$ give
\begin{align*}
\int_0^Tw\,u\,dt
&=w(T)U(T)-\int_0^Tw'(t)U(t)\,dt
 =W(T)\,\ell+R(T),\\
R(T)&=Tw(T)\varepsilon(T)-\int_0^Tw'(t)\,t\,\varepsilon(t)\,dt.
\end{align*}
(b) If $|\varepsilon(t)|\le M\,t^{-\vartheta}$ then, using $w'\ge0$ and one more integration by parts,
\begin{align*}
\int_0^Tw'(t)\,t^{1-\vartheta}\,dt
&=w(T)T^{1-\vartheta}-(1-\vartheta)\int_0^Tw(t)\,t^{-\vartheta}\,dt\\
&\le w(T)\,T^{1-\vartheta},
\end{align*}
so $|R(T)|\le2M\,w(T)\,T^{1-\vartheta}\le2\Lambda M\,W(T)\,T^{-\vartheta}$ for $T\ge T_1$. (a) Given $\epsilon>0$ pick $T_0\ge T_1$ with $\sup_{t\ge T_0}|\varepsilon(t)|\le\epsilon$; then
\begin{align*}
|R(T)|
&\le Tw(T)\,\epsilon+\int_0^{T_0}w'(t)|U(t)-t\ell|\,dt
  +\epsilon\int_{T_0}^Tw'(t)\,t\,dt\\
&\le2\Lambda\,W(T)\,\epsilon+C(T_0),
\end{align*}
and dividing by $W(T)\to\infty$ proves the claim.
\end{proof}

\begin{corollary}[Homogeneous singularities along power-law spirals]\label{cor:hom}
Let $d\ge2$, $x_0\in\R^{d+1}$, $\alpha\ge0$, and
\[
f(x)=|x-x_0|^{-\alpha}\,a\Bigl(\frac{x-x_0}{|x-x_0|}\Bigr),\qquad 0<|x-x_0|\le r_0.
\]
Let $g(t)=r_0(1+t)^{-b}$ with $b>0$, let $\Gamma(t)=x_0+g(t)\Phi_d(\theta_0+t\omega)$ and $W_\alpha(T)=\int_0^Tg(t)^{-\alpha}\,dt$. Then:
\begin{enumerate}[label=\textup{(\alph*)}, leftmargin=2.2em]
\item if $\omega\in\mathcal D(\gamma,\tau)$ and $a\in C^s(\Sph^d)$, $0<s\le1$, then for every $\theta_0$ and every $T\ge1$,
\[
\Bigl|\frac1{W_\alpha(T)}\int_0^Tf(\Gamma(t))\,dt-\langle a\rangle\Bigr|\le C\,\|a\|_{C^s}\,T^{-\vartheta},\qquad C=C(d,\gamma,\tau,s,b,\alpha),
\]
with $\vartheta$ as in Theorem~\ref{thm:main};
\item if $k\cdot\omega\neq0$ for all $k\neq0$ and $a\in C(\Sph^d)$, the left-hand side of the display in \textup{(a)} tends to $0$ as $T\to\infty$, for every $\theta_0$;
\item if $k\cdot\omega\neq0$ for all $k\neq0$ and $a\in L^1(\Sph^d)$, the left-hand side of the display in \textup{(a)} tends to $0$ as $T\to\infty$, for a.e.\ $\theta_0$.
\end{enumerate}
\end{corollary}

\begin{proof}
Along the spiral,
\[
f(\Gamma(t))=w(t)u(t),\qquad
w(t)=r_0^{-\alpha}(1+t)^{b\alpha},\qquad
u(t):=a(\Phi_d(\theta_0+t\omega)).
\]
Here $w$ is positive, $C^1$, and nondecreasing, and $Tw(T)/W(T)\to1+b\alpha$. Hence the constant $\Lambda$ of Theorem~\ref{thm:abel} is finite, with $\Lambda\le C(b,\alpha)$ for $T_1=1$. Applying Theorem~\ref{thm:abel}, part (a) follows from Theorem~\ref{thm:main} with $\ell=\langle a\rangle$ and $M=C\|a\|_{C^s}$; part (b) follows from Remark~\ref{rem:continuous}; and part (c) follows from Theorem~\ref{thm:L1}, whose proof also gives the required local integrability for a.e.\ $\theta_0$. For $\alpha=0$ the statement reduces to Remark~\ref{rem:angular} and Remark~\ref{rem:L1sharp}(c).
\end{proof}

\begin{remark}[The normalization is necessary]\label{rem:norm}
Here $W_\alpha(T)\asymp T^{1+b\alpha}$, so the plain time average satisfies
\[
\frac1T\int_0^Tf(\Gamma(t))\,dt=\frac{W_\alpha(T)}T\bigl(\langle a\rangle+o(1)\bigr)\asymp T^{b\alpha}\,\langle a\rangle,
\]
which diverges whenever $b\alpha>0$ and $\langle a\rangle\neq0$. Thus the normalization by $W_\alpha(T)$ is essential: equivalently, one averages the bounded angular factor $g^\alpha f\circ\Gamma$. With this normalization Corollary~\ref{cor:hom} recovers $\langle a\rangle$ and, in the H\"older-Diophantine regime, the rate of Theorem~\ref{thm:main}.
\end{remark}

For rapidly shrinking profiles even the normalized average fails, as we now show.

\begin{proposition}[Exponential spirals do not average]\label{prop:exp}
Let $d=2$, $\omega=(1,\beta)$ with $\beta$ irrational, and let $g(t)=r_0e^{-\nu t}$, $\nu>0$, $\alpha>0$, so that $w(t)=g(t)^{-\alpha}=r_0^{-\alpha}e^{\varrho t}$ with $\varrho=\alpha\nu$, and $W(T)=\int_0^Tw$. There exists $\varphi\in C^\infty(\Sph^2)$ with $\langle\varphi\rangle=0$, independent of $\beta,\nu,\alpha$, such that for every $\theta_0\in\T^2$ the normalized averages
\[
h(T)=\frac1{W(T)}\int_0^Tw(t)\,\varphi\bigl(\Phi_2(\theta_0+t\omega)\bigr)\,dt
\]
have no limit as $T\to\infty$. In particular, for $f=|x-x_0|^{-\alpha}\varphi\bigl(\frac{x-x_0}{|x-x_0|}\bigr)$ the normalized spiral means along the exponential spiral diverge for every phase, although the normalized spherical means $r^{\alpha}\langle f_r\rangle\equiv\langle\varphi\rangle=0$ are constant in $r$.
\end{proposition}

\begin{proof}
Fix $\zeta\in C_c^\infty((0,1))$, $\zeta\not\equiv0$, and set $F(u_1,u_2)=\zeta(u_2)\cos(2\pi u_1)$ on $\T^2$. Since $F$ vanishes for $u_2$ outside a compact subset of $(0,1)$, and since $\Phi_2$ restricts to a smooth diffeomorphism from $\T^1\times(0,1)$ onto $\Sph^2$ minus the poles (its latitude being $z_2(u)=2u-1$), the function $\varphi:=F\circ(\Phi_2|_{\T^1\times(0,1)})^{-1}$, extended by $0$ near the poles, belongs to $C^\infty(\Sph^2)$, satisfies $\varphi\circ\Phi_2=F$ and $\langle\varphi\rangle=\widehat F(0,0)=0$ by Lemma~\ref{lem:push}. The nonzero Fourier coefficients of $F$ are $\widehat F(\pm1,m)=\frac12\widehat\zeta(m)$, $m\in\Z$, an absolutely summable family, and $\widehat\zeta(m_*)\neq0$ for some $m_*$.

Fix $\theta_0$. Then $u(t):=F(\theta_0+t\omega)=\sum_k\widehat F(k)\,e(k\cdot\theta_0)\,e(\mu_kt)$ with $\mu_k=k\cdot\omega$, an absolutely convergent expansion, and termwise integration (justified by $\sum_k|\widehat F(k)|\int_0^Tw<\infty$) gives, with $\varrho=\alpha\nu$,
\[
h(T)=\sum_k\widehat F(k)\,e(k\cdot\theta_0)\,\frac{\varrho\bigl(e^{\varrho T}e(\mu_kT)-1\bigr)}{(e^{\varrho T}-1)(\varrho+2\pi i\mu_k)}
=H(T)+\sum_kb_k\,\frac{e(\mu_kT)-1}{e^{\varrho T}-1},
\]
where
\[
H(T)=\sum_kb_k\,e(\mu_kT),\qquad b_k=\widehat F(k)\,e(k\cdot\theta_0)\,\frac{\varrho}{\varrho+2\pi i\mu_k},\qquad |b_k|\le|\widehat F(k)|.
\]
The second sum has modulus at most $2\sum_k|\widehat F(k)|/(e^{\varrho T}-1)\to0$, so $h(T)=H(T)+o(1)$. For $k=(\pm1,m)$ in the spectrum of $F$ the frequencies $\mu_k=\pm1+m\beta$ are nonzero and pairwise distinct, because $\beta$ is irrational. Suppose $h(T)\to\ell$; then $H(T)\to\ell$, hence $\frac1S\int_0^S|H(T)-\ell|^2\,dT\to0$. On the other hand, expanding the square and using absolute convergence together with $\frac1S\int_0^Se(\mu T)\,dT\to0$ for every real $\mu\neq0$ (Lemma~\ref{lem:osc}), applied to the frequencies $\mu_k$ and to the pairwise differences $\mu_k-\mu_l\neq0$, $k\neq l$,
\[
\frac1S\int_0^S|H(T)-\ell|^2\,dT\longrightarrow\sum_k|b_k|^2+|\ell|^2.
\]
Hence $b_k=0$ for all $k$, contradicting $|b_{(1,m_*)}|=\frac12|\widehat\zeta(m_*)|\,\varrho/|\varrho+2\pi i\mu_{(1,m_*)}|>0$. Thus $h$ has no limit. (For the background on uniformly almost periodic functions underlying this computation see \cite{Besicovitch}.)
\end{proof}

\begin{remark}[Power-law versus exponential profiles]\label{rem:dich}
Theorem~\ref{thm:abel} applies to $w=g^{-\alpha}$ whenever $Tg(T)^{-\alpha}/\int_0^Tg^{-\alpha}$ stays bounded, which holds for every power-law profile. Proposition~\ref{prop:exp} shows that for exponential profiles the conclusion itself can fail, for every phase and already for $C^\infty$ mean-zero data. These two model regimes demonstrate that the radial concentration rate interacts genuinely with angular equidistribution. The exact borderline class of admissible profiles is posed as an open problem in Section~\ref{sec:open}.
\end{remark}

\section{Limits of uniformity: no rate over all rationally independent frequencies}\label{sec:opt}

Throughout this section $d=2$, $\theta_0=0$, and for $\omega\in\R^2$ with rationally independent entries we write $A_T(\varphi)=\frac1T\int_0^T\varphi(\Phi_2(t\omega))\,dt$. By Remark~\ref{rem:continuous}, $A_T(\varphi)\to\langle\varphi\rangle$ for every continuous $\varphi$. The next theorem shows that this convergence admits no modulus valid for all rationally independent frequencies, even on the unit ball of mean-zero smooth observables in $C^1$.

\begin{remark}[Convention on the norms $\|\cdot\|_{C^j}$]\label{rem:normconv}
In this section $\|\cdot\|_{C^1(\Sph^2)}$ always denotes the norm $\|\cdot\|_\infty+[\cdot]_1$ of the Notation paragraph in Section~\ref{sec:intro}, i.e.\ the case $s=1$ of $\|\cdot\|_{C^s}$, taken with respect to the Euclidean chordal metric; this is the same norm in which Theorem~\ref{thm:main} measures its data, so that Theorems~\ref{thm:main} and~\ref{thm:opt} refer to one and the same class of observables. For integers $j\ge2$ we fix, once and for all, any family of norms $\|\cdot\|_{C^j(\Sph^2)}$ on $C^\infty(\Sph^2)$ which
\begin{enumerate}[label=\textup{(\alph*)}, leftmargin=2.2em]
\item is nondecreasing in $j$, with $\|\cdot\|_{C^1}\le\|\cdot\|_{C^2}$, and
\item defines the $C^\infty$ topology on $\Sph^2$, in the sense that a sequence which is Cauchy for every $\|\cdot\|_{C^j}$ converges in $C^\infty(\Sph^2)$.
\end{enumerate}
For instance, one may take $\|\varphi\|_{C^j}=\|\varphi\|_\infty+[\varphi]_1+\sum_{2\le i\le j}\|\nabla^i\varphi\|_\infty$ with $\nabla$ the covariant derivative of the round metric. Only \textup{(a)} and \textup{(b)} are used below; no other property of the family is invoked. Note that the chordal seminorm cannot simply be replaced by the gradient sup-norm without altering constants: since the chordal and geodesic distances on $\Sph^2$ satisfy $\theta\le\frac\pi2|x-y|$, one has $[\varphi]_1\le\frac\pi2\|\nabla\varphi\|_\infty$, and the factor $\frac\pi2$ is attained in the limit of antipodal pairs. Fixing $\|\cdot\|_{C^1}$ as above is what makes the normalization $\|\varphi\|_{C^1}\le1$ in Theorem~\ref{thm:opt} exact rather than up to a constant.
\end{remark}

\begin{theorem}[No uniform rate]\label{thm:opt}
Let $R\colon[1,\infty)\to(0,\infty)$ be any function with $\lim_{T\to\infty}R(T)=0$. Then there exist $\beta\in(0,1)$ irrational (so that $\omega=(1,\beta)$ is rationally independent) and $\varphi\in C^\infty(\Sph^2)$ with
\[
\langle\varphi\rangle=0,\qquad \|\varphi\|_{C^1(\Sph^2)}\le1,
\qquad
\limsup_{T\to\infty}\frac{|A_T(\varphi)|}{R(T)}=\infty.
\]
\end{theorem}

The proof occupies the rest of the section.

\subsection*{Building blocks}
Fix once and for all $\chi\in C_c^\infty((0,1))$ with $0\le\chi\le1$ and $\int_0^1\chi\ge\frac78$; regard $\chi$ as a smooth function on $\T^1$ vanishing near $0$. Its Fourier coefficients $\widehat\chi(\nu)$, $\nu\in\Z$, decay rapidly; set
\begin{align*}
c_0&=\tfrac12\widehat\chi(0)\ge\tfrac7{16},
&S_\chi&=\sum_{\nu\in\Z}|\widehat\chi(\nu)|<\infty,\\
S'_\chi&=\sum_{\nu\in\Z}\frac{|\widehat\chi(\nu)|}{\max(1,|\nu|)},
&S''_\chi&=\sup_{\nu\neq0}|\nu|^2|\widehat\chi(\nu)|<\infty.
\end{align*}
For coprime integers $p,q\ge1$ define $F_{p,q}\in C^\infty(\T^2)$ by
\[
F_{p,q}(u_1,u_2)=\chi(u_2)\cos\bigl(2\pi(pu_1-qu_2)\bigr),
\]
whose Fourier support is contained in $\{(\pm p,m):m\in\Z\}$ with
\begin{equation}\label{eq:blockcoeff}
\widehat F_{p,q}(p,m)=\tfrac12\widehat\chi(m+q),\qquad \widehat F_{p,q}(-p,m)=\tfrac12\widehat\chi(m-q).
\end{equation}
In particular $\widehat F_{p,q}(0,0)=0$. Since $d=2$, the latitude in Definition~\ref{def:map} is $z_2(u)=2u-1$ and $\Phi_2$ restricts to a smooth diffeomorphism from $\T^1\times(0,1)$ onto $\Sph^2$ minus the poles; as $F_{p,q}$ vanishes for $u_2$ outside $\supp\chi\Subset(0,1)$, the function
\[
\varphi_{p,q}:=F_{p,q}\circ\bigl(\Phi_2|_{\T^1\times(0,1)}\bigr)^{-1},\qquad\text{extended by $0$ near the poles,}
\]
belongs to $C^\infty(\Sph^2)$, satisfies $\varphi_{p,q}\circ\Phi_2=F_{p,q}$, $\|\varphi_{p,q}\|_\infty\le1$, and $\langle\varphi_{p,q}\rangle=\widehat F_{p,q}(0,0)=0$ by Lemma~\ref{lem:push}.

For $\omega=(1,\beta)$, the absolutely convergent Fourier expansion gives
\begin{equation}\label{eq:blockexp}
\begin{aligned}
A_T(\varphi_{p,q})=\sum_{m\in\Z}\Bigl[&\widehat F_{p,q}(p,m)\,I_T\,e\bigl((p+m\beta)t\bigr)\\
&+\widehat F_{p,q}(-p,m)\,I_T\,e\bigl((-p+m\beta)t\bigr)\Bigr],
\end{aligned}
\end{equation}
where $I_T\,e(at):=\frac1T\int_0^Te(at)\,dt$. The pair of principal modes $m=\mp q$ has frequency $\pm\delta'$, $\delta'=p-q\beta$, and contributes, by \eqref{eq:blockcoeff},
\begin{equation}\label{eq:principal}
2c_0\cdot\mathrm{Re}\,I_T\,e(\delta't)=2c_0\,\frac{\sin(2\pi\delta'T)}{2\pi\delta'T}\ge2c_0\cdot0.9\qquad\text{whenever }2\pi|\delta'|T\le\frac\pi4,
\end{equation}
since $x\mapsto\sin x/x$ is even and decreasing on $[0,\pi]$, so that $\sin x/x\ge\sin(\pi/4)/(\pi/4)=2\sqrt2/\pi=0.9003\ldots>0.9$ for $|x|\le\pi/4$.

\subsection*{The recursive construction}
We construct $\beta$ through its continued fraction $\beta=[0;a_1,a_2,\dots]$, choosing the partial quotients $a_j$ recursively; $p_j/q_j$ denote the convergents ($p_0=0$, $q_0=1$, $p_1=1$, $q_1=a_1$, $p_{j+1}=a_{j+1}p_j+p_{j-1}$, $q_{j+1}=a_{j+1}q_j+q_{j-1}$), which satisfy $\gcd(p_j,q_j)=1$ and
\begin{equation}\label{eq:cf}
\Bigl|\beta-\frac{p_j}{q_j}\Bigr|\le\frac1{q_jq_{j+1}}\qquad(j\ge1),
\end{equation}
a bound determined by $a_1,\dots,a_{j+1}$ alone \cite[Ch.~I]{Cassels}. Set $a_1=1$, so $p_1=q_1=1$ and $\beta\in(\frac12,1)$, i.e.\ $\beta\ge\beta_*:=\frac12$. Alongside we choose weights $\varepsilon_j>0$ and times $T_j\uparrow\infty$; the conventional values $\varepsilon_0=1$, $T_0=1$ serve only to initialize the recursion, no observable $\varphi_0$ being defined, and all sums over blocks below start at index $1$. Suppose $a_1,\dots,a_j$ (hence $p_j,q_j$ and $\varphi_j:=\varphi_{p_j,q_j}$) and all earlier $\varepsilon_i,T_i$ are chosen.

\emph{Step 1.} Set $\displaystyle\varepsilon_j=\min\Bigl(\frac{2^{-j}}{1+\|\varphi_j\|_{C^j(\Sph^2)}},\ \frac{c_0\,\varepsilon_{j-1}}{32}\Bigr)$, the norms $\|\cdot\|_{C^j(\Sph^2)}$ being those fixed in Remark~\ref{rem:normconv}.

\emph{Step 2.} Choose $T_j\ge\max(T_{j-1}+j,\,q_j,\,q_j^j)$ so large that, in addition,
\begin{equation}\label{eq:Tconditions}
R(T_j)\le\frac{c_0\varepsilon_j}{4^j},\qquad
\frac{4S'_\chi}{\pi\beta_*T_j}\le\frac{c_0}8,\qquad
\frac{4q_jS_\chi/\pi+4S''_\chi}{T_j}\le\frac{c_0\varepsilon_j}{16}.
\end{equation}
This is possible since $R(T)\to0$ and all other quantities are fixed at this stage.

\emph{Step 3.} Set $a_{j+1}=\lceil8T_j/q_j\rceil$ ($\ge8$), so that
\begin{equation}\label{eq:qgrowth}
8T_j\le q_{j+1}=a_{j+1}q_j+q_{j-1}\le8T_j+2q_j\le10T_j.
\end{equation}

This defines $\beta$ and, with it, $\varphi=\sum_{j\ge1}\varepsilon_j\varphi_j$. For every $N$,
\[
\sum_j\varepsilon_j\|\varphi_j\|_{C^N}
\le\sum_{j<N}\varepsilon_j\|\varphi_j\|_{C^N}+\sum_{j\ge N}2^{-j}<\infty
\]
by Step~1 together with the monotonicity \textup{(a)} of Remark~\ref{rem:normconv}, which gives $\|\varphi_j\|_{C^N}\le\|\varphi_j\|_{C^j}$ for $j\ge N$; hence the series converges in every $\|\cdot\|_{C^N}$, and $\varphi\in C^\infty(\Sph^2)$ by \textup{(b)}. Moreover, using $\|\varphi_j\|_{C^1}\le\|\varphi_j\|_{C^j}$ for every $j\ge1$,
\[
\|\varphi\|_{C^1}
\le\sum_{j\ge1}\varepsilon_j\|\varphi_j\|_{C^1}
\le\sum_{j\ge1}2^{-j}=1.
\]
We also have $\langle\varphi\rangle=0$ and, by uniform convergence, $A_T(\varphi)=\sum_j\varepsilon_jA_T(\varphi_j)$ for every $T$. Since all $a_j\ge1$ and infinitely many steps occur, $\beta$ is irrational.

\subsection*{Three estimates at time $T_j$}
Fix $j$ and decompose
\[
A_{T_j}(\varphi)=\varepsilon_jA_{T_j}(\varphi_j)
+\sum_{1\le i<j}\varepsilon_iA_{T_j}(\varphi_i)
+\sum_{i>j}\varepsilon_iA_{T_j}(\varphi_i).
\]

\emph{(1) The own term.} Let $\delta_j=|p_j-q_j\beta|$. By \eqref{eq:cf} and \eqref{eq:qgrowth}, $\delta_j\le q_j\cdot\frac1{q_jq_{j+1}}\le\frac1{8T_j}$, hence $2\pi\delta_jT_j\le\frac\pi4$ and the principal modes contribute at least $1.8\,c_0$ by \eqref{eq:principal}. The remaining modes of $\varphi_j$ in \eqref{eq:blockexp} have $m=\mp q_j+l$, $l\neq0$, and frequencies $\pm\delta'_j+l\beta$ with $|\pm\delta'_j+l\beta|\ge\beta_*|l|-\delta_j\ge\beta_*|l|-\frac18\ge\frac12\beta_*|l|$; their coefficients are $\frac12\widehat\chi(\pm l)$, so by Lemma~\ref{lem:osc} their total contribution is at most
\[
2\sum_{l\neq0}\frac{|\widehat\chi(l)|}2\cdot\frac2{\pi T_j\beta_*|l|}
\le\frac{2S'_\chi}{\pi\beta_*T_j}\le\frac{4S'_\chi}{\pi\beta_*T_j}\le\frac{c_0}8
\]
by the second condition in \eqref{eq:Tconditions}. Hence $|A_{T_j}(\varphi_j)|\ge1.8c_0-\frac{c_0}8$.

\emph{(2) Earlier terms are non-resonant (anti-circularity).} Let $1\le i<j$. The frequencies occurring in $A_{T_j}(\varphi_i)$ are $\pm p_i+m\beta$, $m\in\Z$. We claim
\begin{equation}\label{eq:anticirc}
|\pm p_i+m\beta|\ge\frac1{2q_j}\qquad\text{for all }|m|\le\tfrac12q_{j+1}.
\end{equation}
Indeed $p_iq_j\pm mp_j\neq0$ for every $m\in\Z$: otherwise $p_j\mid p_iq_j$, and $\gcd(p_j,q_j)=1$ would force $p_j\mid p_i$, impossible because $0<p_i<p_j$ (the numerators are strictly increasing: $p_2=a_2\ge8>1=p_1$ and $p_{j+1}\ge p_j+p_{j-1}$). Hence $|\pm p_i+m\frac{p_j}{q_j}|\ge\frac1{q_j}$, while $|m|\,|\beta-\frac{p_j}{q_j}|\le\frac{q_{j+1}}2\cdot\frac1{q_jq_{j+1}}=\frac1{2q_j}$ by \eqref{eq:cf}; the triangle inequality gives \eqref{eq:anticirc}. Note that \eqref{eq:anticirc} covers in particular the principal modes of $\varphi_i$: no circularity arises, because the bound holds for all $m$ simultaneously. By Lemma~\ref{lem:osc} and \eqref{eq:blockcoeff}, the modes with $|m|\le\frac12q_{j+1}$ contribute to $|A_{T_j}(\varphi_i)|$ at most $S_\chi\cdot\frac{2q_j}{\pi T_j}$. For the tail $|m|>\frac12q_{j+1}$ we use the trivial bound $|I_T\,e(at)|\le1$: since $q_i\le q_j\le\frac18q_{j+1}$, the relevant indices satisfy $|m\mp q_i|\ge\frac12q_{j+1}-\frac18q_{j+1}\ge\frac14q_{j+1}\ge2T_j$, so the tail is bounded by
\[
\sum_{|\nu|\ge2T_j}|\widehat\chi(\nu)|\le S''_\chi\sum_{|\nu|\ge2T_j}|\nu|^{-2}\le\frac{2S''_\chi}{2T_j-1}\le\frac{2S''_\chi}{T_j}
\]
for $T_j\ge1$. Summing over $1\le i<j$ and using $\sum_{1\le i<j}\varepsilon_i\le\sum_{i\ge1}2^{-i}=1$ together with the third condition in \eqref{eq:Tconditions},
\begin{align*}
\sum_{1\le i<j}\varepsilon_i|A_{T_j}(\varphi_i)|
&\le\frac{2q_jS_\chi}{\pi T_j}\sum_{1\le i<j}\varepsilon_i
 +\frac{2S''_\chi}{T_j}\sum_{1\le i<j}\varepsilon_i\\
&\le\frac{2q_jS_\chi/\pi+2S''_\chi}{T_j}
\le\frac{4q_jS_\chi/\pi+4S''_\chi}{T_j}
\le\frac{c_0\varepsilon_j}{16}.
\end{align*}

\emph{(3) Later terms are small.} By Step~1,
$\varepsilon_i\le(c_0/32)\varepsilon_{i-1}\le\varepsilon_{i-1}/32$ for all $i$. Hence
\[
\sum_{i>j}\varepsilon_i
\le\frac{c_0\varepsilon_j}{32}\sum_{\ell\ge0}32^{-\ell}
\le\frac{c_0\varepsilon_j}{16}.
\]
Since $\|\varphi_i\|_\infty\le1$, it follows that
\[
\sum_{i>j}\varepsilon_i|A_{T_j}(\varphi_i)|
\le\frac{c_0\varepsilon_j}{16}.
\]

\subsection*{Conclusion}
Combining the three estimates,
\begin{align*}
|A_{T_j}(\varphi)|
&\ge\varepsilon_j\Bigl(1.8c_0-\frac{c_0}8\Bigr)
 -\frac{c_0\varepsilon_j}{16}-\frac{c_0\varepsilon_j}{16}\\
&\ge c_0\,\varepsilon_j\ge4^jR(T_j)
\end{align*}
by the first condition in \eqref{eq:Tconditions}. Since $T_j\to\infty$,
\[
\limsup_{T\to\infty}\frac{|A_T(\varphi)|}{R(T)}
\ge\limsup_{j\to\infty}4^j=\infty.
\]
This proves Theorem~\ref{thm:opt}.\hfill$\square$

\begin{remark}\label{rem:opt}
The frequency $\omega=(1,\beta)$ produced above is rationally independent, so the flow is uniquely ergodic and $A_T(\varphi)\to0=\langle\varphi\rangle$. Theorem~\ref{thm:opt} says that no prescribed modulus $R(T)\to0$ is valid uniformly over the full class of rationally independent frequencies, even for smooth mean-zero observables satisfying $\|\varphi\|_{C^1}\le1$.

The additional requirement $T_j\ge q_j^j$ makes the arithmetic nature of the construction explicit. Indeed, \eqref{eq:qgrowth} gives $q_{j+1}\ge8q_j^j$, and hence by \eqref{eq:cf}
\[
\left|\beta-\frac{p_j}{q_j}\right|\le\frac1{q_jq_{j+1}}\le\frac1{8q_j^{j+1}}.
\]
Thus $\beta$ is a Liouville number and in particular belongs to none of the Diophantine classes $\mathcal D(\gamma,\tau)$. Theorems~\ref{thm:main} and~\ref{thm:opt} therefore illustrate complementary mechanisms: quantitative small-divisor control yields an explicit polynomial rate, whereas over the unrestricted class of rationally independent frequencies no universal modulus is possible. The theorem does not assert that every non-Diophantine frequency has arbitrarily slow convergence.
\end{remark}

\section{Open problems}\label{sec:open}

\begin{problem}[Higher smoothness]\label{prob:smooth}
For angular data smoother than Lipschitz --- say $\varphi\in C^{1,s'}(\Sph^d)$ with $s'\in(0,1]$, or $\varphi$ in a Sobolev class $H^\varsigma(\Sph^d)$ with $\varsigma>1$, neither of which is covered by the scale $C^s$, $0<s\le1$, fixed in the Notation --- our method saturates: the cutoff $X$ costs $\eta^{-s}$ but the polar visits only gain $\eta$, and the exponent $\vartheta=\frac{s}{(1+s)(\tau+d)}$ is maximized as $s\uparrow1$ within our proof. Does additional smoothness beyond Lipschitz improve the rate, e.g.\ via spherical-harmonic decompositions replacing the composition with $\Phi_d$ altogether? Two benchmarks are natural here. The first is the Sobolev-type rate for QMC designs on spheres \cite{BrauchartGrabner}. The second is the Koksma--Hlawka route of Remark~\ref{rem:HK}. There the mixed derivative $\partial_{u_1}\cdots\partial_{u_d}(\varphi\circ\Phi_d)$ is integrable on the fundamental domain as soon as $\varphi$ has $d$ bounded derivatives: the polar factors $w(u_\ell)^{\frac1\ell-1}$ are integrable in their own variables, so the condition is one on $\varphi$ alone --- but it is a condition of order $d$, and not of order $2$, since each of the $d$ coordinates is differentiated once and the chain rule places up to $d$ derivatives on $\varphi$. (For $d=2$ the sufficient criterion used here is, for instance, $\varphi\in C^2$; more refined bounded-variation criteria may lower this threshold. For $d\ge3$ the hypothesis $\varphi\in C^{1,s'}$ alone is \emph{not} sufficient for the present mixed-derivative argument.) Consequently, for $\varphi\in C^{d}(\Sph^d)$ the discrepancy of the linear flow \cite{Borda,DumasFischler} applies with no mollification and no cutoff, and \eqref{eq:HKrate} improves to $T^{-\frac1{\tau+1}+\varepsilon}$, since one may take $\epsilon\asymp1$. Determining the exact smoothness threshold $\varsigma=\varsigma(d)$ at which finite Hardy--Krause variation first becomes available --- our sufficient criterion gives $\varsigma\le d$, and it is presumably not optimal --- and interpolating between $\vartheta_{\mathrm{HK}}$ and $\frac1{\tau+1}$ across the range $s\le\varsigma$, would complete the picture begun in Remark~\ref{rem:HK}.
\end{problem}

\begin{problem}[Sharpness of the exponent for Diophantine frequencies]\label{prob:sharp}
Theorem~\ref{thm:opt} shows no uniform rate exists over all rationally independent $\omega$. For fixed $\omega\in\mathcal D(\gamma,\tau)$, what is the optimal exponent $\vartheta^*(s,\tau,d)$? Lower-bound constructions along convergents, in the spirit of Section~\ref{sec:opt} but within the Diophantine class, should give upper bounds on $\vartheta^*$; we do not expect $\frac{s}{(1+s)(\tau+d)}$ to be sharp. As explained in Remark~\ref{rem:singlescale}, the loss in our argument is located in the scale of the cutoff rather than in the geometric exponent of Lemma~\ref{lem:reg}(b); at the Lipschitz endpoint $s=1$ the bound $|\widehat F(k)|\lesssim|k|_\infty^{-1}$ recorded there is already available with no cutoff at all, and what is missing is a substitute for the uniform-norm control of the truncation error. Supplying one --- for instance an occupation-time estimate for the high-frequency tail --- would improve the exponent at $s=1$ from $\frac1{2(\tau+d)}$ to $\frac1{2\tau+2d-1}$, by the heuristic computation carried out in Remark~\ref{rem:singlescale}(b); making that computation rigorous, which requires a quantitative rate for Fej\'er convergence away from the cut, is the first concrete step. Remark~\ref{rem:jump} shows that for $d=2$ the jump-subtraction route reaches $\frac{s}{2(\tau+2)}$ and no further, so it is not by itself such a substitute; whether its recursive version reaches anything better than $\vartheta$ for $d\ge3$ is open as well. Remark~\ref{rem:localization} locates the loss quantitatively: on data vanishing near $\Sigma$, where no cutoff is needed, the same lemmas give $\frac{s}{\tau+d}$; the factor $1+s$ separating it from $\vartheta$ measures exactly what a sharper treatment of the polar collar would have to recover. Finally, $\vartheta^*$ is bounded below by the maximum of $\vartheta$ and the mollified Koksma--Hlawka exponent $\vartheta_{\mathrm{HK}}=\frac{s}{d(\tau+1)}$ of Remark~\ref{rem:HK}, the two crossing at $\tau(d-1-s)=ds$; a genuinely sharp result should explain that crossing rather than take the maximum, and in particular should say whether $\frac1{\tau+1}$ --- the exponent available for $C^{d}$ data, and the benchmark supplied by the Diophantine discrepancy estimate used here --- is attained in the limit $s\uparrow1$ in low dimension.
\end{problem}

\begin{problem}[Principal-value kernels]
Corollary~\ref{cor:hom} handles homogeneous singularities $f=|x-x_0|^{-\alpha}a$ via the natural normalization $W_\alpha$. The genuinely singular case is that of kernels with cancellation, $\langle a\rangle=0$: then the normalized limit vanishes, and the question becomes whether spiral averages, under a second-order normalization, reproduce principal-value integrals $\mathrm{p.v.}\!\int f$ --- for instance Calder\'on--Zygmund convolution kernels evaluated along the trajectory. This requires genuinely singular-integral technology along the curve and is deferred to future work.
\end{problem}

\begin{problem}[Admissible radial profiles]
By Theorem~\ref{thm:abel} and Proposition~\ref{prop:exp}, normalized spiral means for singular data converge for all power-law profiles and fail for exponential ones. Characterize the exact class of profiles $g$ (equivalently, of weights $w=g^{-\alpha}$) for which convergence holds for all continuous angular data and all non-resonant frequencies. Is the condition $\limsup_{T\to\infty}Tw(T)/W(T)<\infty$ of Theorem~\ref{thm:abel} also necessary?
\end{problem}

\begin{problem}[Optimality in higher dimensions]
Extend Theorem~\ref{thm:opt} to all $d\ge2$, and to an arbitrary initial phase. The building blocks and the continued-fraction mechanism are two-dimensional; for $d\ge3$ one expects an analogous construction using simultaneous approximation, with the additional freedom helping rather than hindering, but the anti-circularity estimate \eqref{eq:anticirc} needs a replacement. The restriction $\theta_0=0$ is of a different nature and should be easier to remove: for a general phase the two principal modes contribute $2c_0\,\mathrm{Re}\bigl(e(k\cdot\theta_0)\,I_T\,e(\delta't)\bigr)$, which is bounded below only when the phase does not rotate the resonant contribution to the imaginary axis; adding to each block a second building block in quadrature, or shifting $T_j$ by a bounded amount, should restore the lower bound uniformly in $\theta_0$.
\end{problem}

\section*{Declarations}

\subsection*{Funding}
The author received no funding for this work.

\subsection*{Competing interests}
The author declares no competing interests.

\subsection*{Data availability}
Not applicable: no datasets were generated or analysed during this study.

\end{document}